\documentclass[10pt]{amsart}
\usepackage{amscd, amsfonts, amsthm, amsgen, amsmath, amssymb ,verbatim, enumerate, textcomp}
\usepackage{hyperref} 
\usepackage[cmtip, all]{xy}
\usepackage{graphicx}
\usepackage[margin=1in]{geometry}

\newtheorem{theorem}{Theorem}[section]

\newtheorem{lemma}[theorem]{Lemma}
\newtheorem{corollary}[theorem]{Corollary}

\theoremstyle{definition}
\newtheorem{definition}[theorem]{Definition}

\newtheorem{examples}[theorem]{Examples}

\theoremstyle{remark}
\newtheorem{remark}[theorem]{Remark}

\theoremstyle{keyobservation}

\newtheorem{remarks}[theorem]{Remarks}
\newtheorem{shypothesis}[theorem]{Standing Hypothesis}
\numberwithin{equation}{section}

\numberwithin{equation}{subsection}
\newcommand{\be}%
  {\protect\setcounter{equation}{\value{subsubsection}}}  
  \newcommand{\ee}%
   {\protect\setcounter{subsubsection}{\value{equation}}}

  {\protect\setcounter{subsubsection}{\value{equation}}}

\numberwithin{equation}{subsection}

\newcommand{\GL}{\operatorname{GL}}

\newcommand{\bG}{\mathbf G}

\newcommand{\bZ}{\mathbf Z}

\newcommand{\cN}{\mathcal N}

\newcommand{\cX}{\mathcal X}

\newcommand{\bK}{\mathbf K}

\renewcommand{\L}{\mathcal L}

\newcommand{\Z}{\mathcal Z}

\renewcommand{\O}{\mathcal{O}}

\def\displaytimes_#1{\mathrel{\mathop{\times}\limits_{#1}}}

\def\displayotimes_#1{\mathrel{\mathop{\bigotimes}\limits_{#1}}}

 \def\ari[#1]{\ar@{^(->}[#1]}
 \def\are[#1]{\ar[#1]^{\txt{\'et}}}
 \def\areh[#1]{\ar[#1]|{\txt{$H$-eq}}^{\txt{\'et}}}
 \def\ars[#1]{\ar@{->>}[#1]}
 \newcommand{\dplus}{\ar@{}[d]|{\mbox{$\oplus$}}}
 \newcommand{\dtimes}{\ar@{}[d]|{\mbox{$\times$}}}

\def \rmB{\rm B}
\def \rmBG{\rm {BG}}

\def \Cl{\mathbb C}

\def \colimK.{\underset {\underset K^.  \rightarrow}  {\hbox {lim}}}

\def \colimU.{\underset {\underset U_.  \rightarrow}  {\hbox {lim}}}

\def \cosimp1{\stackrel{\rightrightarrows}{ \leftarrow}}

\def \compl{\, \, {\widehat {}}}

\def \bDX{{\mathbb D}^{\rmG}_{\rmX}}

\def \rmE{\rm E}

\def \EG1{E{(G \times {\mathbb C}^*)}{\underset {G\times {\mathbb C}^*} \to \times}}

\def \EZ(s)1{E{(Z(s) \times {\mathbb C}^*)}{\underset {(Z(s)\times {\mathbb C}^*)} \to \times}}

\newcommand{\eps}{\boldsymbol\varepsilon}

\def \eps{\ \epsilon \ }
\def \EM(u){EM(u){\underset {M(u)} \to \times}}
\def \EM(us){EM(u,s){\underset {M(u, s)} \to \times}}

\def \bG{{\mathbf G}}

\def \rmG{\rm G}

\def\hlimD2{\mathop{\textrm{holim}}\limits_{\Delta_{\le m_2} }}

\def \H{\mathbb H}
\def \rmH{\rm H}
\def \bH{\bf H}

\def \holimm {\underset {\infty \leftarrow m}  {\hbox {holim}}}

\def \invlim1{\underset {\infty \leftarrow q} \to {\hbox {lim}}^1}
\def \rmI{\rm I}
\def \rmi{\rm i}

\def \rmj{\rm j}

\def \rmK{\rm K}
\def \k{\it k}
\def \bKH{\bf KH}

\def \L3{\Lambda \times \Lambda \times \Lambda}
\def \L2{\Lambda \times \Lambda}

\def \limm{\underset {\infty \leftarrow m}  {\hbox {lim}}}

\def \longright2arrow{{\overset \longrightarrow \to {\overset {} \to \longrightarrow}}}

\def \L{L\times \Cl ^*}

\def \rmM{\rm M}

\def \O{{\mathcal O}}

\def \rmp{\rm p}

\def \Q{{\mathbb  Q}}

\def \ra{\rightarrow}
\def \Ra{\Rightarrow}
\def \RG^{R(G)^{\hat {}}\ }

\def \res{respectively}

\def \rmR{\rm R}

\def \rmS{\rm S}

\def \Speck{{\rm Spec}\, {\it k}}

\def \topGcoh*{^{top, *} _{G}}
\def \topGho*{ _{top,*} ^{G}}

\def \rmT{\rm T}

\def \rmU{\rm U}

\def \rmV{\rm V}

\def \rmX{\rm X}

\def \cX{\mathcal X}

\def \rmY{\rm Y}

\def \Z(s){Z(s) \times {\mathbb C}^*}
\def \Z{\mathbb Z}

\def \rmZ{\rm Z}
\def \bZ{\mathbb Z}

\begin{document}

\title{Equivariant Algebraic K-Theory and Derived completions III: Applications }
\author{Gunnar Carlsson}
\address{Department of Mathematics, Stanford University, Building 380, Stanford,
California 94305}
\email{gunnar@math.stanford.edu}
\thanks{  }  
\author{Roy Joshua}
\address{Department of Mathematics, Ohio State University, Columbus, Ohio,
43210, USA.}
\email{joshua.1@math.osu.edu}
\author{Pablo Pelaez}
\address{Instituto de Matem\'aticas, Ciudad Universitaria, UNAM, DF 04510, M\'exico.}
\email{pablo.pelaez@im.unam.mx}


\thanks{AMS Subject classification: 19E08, 14C35, 14L30.  The second author thanks the Simons Foundation for support.}

\maketitle
 
\begin{abstract}
\vskip .1cm
In the present paper, we discuss applications of the derived completion theorems proven in our previous two papers.
One of the main applications is to Riemann-Roch problems for forms of higher equivariant K-theory, which we are able to establish in
great generality both for equivariant G-theory and equivariant homotopy K-theory with respect to actions of linear algebraic groups on
normal quasi-projective schemes over a given field.  We show such Riemann-Roch theorems
apply to all toric and spherical varieties.
\vskip .1cm
We also obtain Lefschetz-Riemann-Roch
theorems involving the fixed point schemes with respect to actions of diagonalizable group schemes.  
We also show the existence of 
certain spectral sequences that compute the homotopy groups of the derived completions of equivariant G-theory starting with equivariant Borel-Moore 
motivic cohomology.
\end{abstract} 
\setcounter{tocdepth}{1}
\tableofcontents\setcounter{tocdepth}{1}
\markboth{Gunnar Carlsson, Roy Joshua and Pablo Pelaez}{Equivariant Algebraic K-Theory and Derived completions III:Applications}
\input xypic
\vfill \eject
\section{Introduction and statement of results}
\vskip .2cm

\vskip .1cm
 \vskip .2cm
 This paper is a sequel to the papers \cite{CJ23} and \cite{CJP23} and is devoted to applications of the derived completion theorems proven in the above two papers.
 In fact we discuss several applications. One class of applications we discuss in detail is to various different forms of 
 Equivariant Riemann-Roch theorems: a significant part of the paper is devoted to these. In fact, it may be important to point out that 
 the derived completion theorems proved in \cite{CJ23} and \cite{CJP23} are such that strong forms of various equivariant Riemann-Roch theorems
 can be derived as important consequences. The main observation here is that one can prove 
Riemann-Roch theorems first as natural transformations, with source the G-theory of the category of equivariant coherent sheaves 
(or the Homotopy K-theory of the category of equivariant perfect complexes) and taking values in the corresponding G-theory (or the Homotopy K-theory)
on a suitable Borel construction associated to the given action of a linear algebraic group. 
\vskip .2cm
Once this is done, various other forms of equivariant 
Riemann-Roch theorems can be obtained. For example, we compose the above Riemann-Roch transformations with 
Riemann-Roch transformations into various forms of Borel-Moore homology theories on the Borel construction: since the source and the target 
of these latter Riemann-Roch transformations are both defined on the Borel construction, these can be deduced from the corresponding
non-equivariant Riemann-Roch transformation from the G-theory or the homotopy K-theory of a scheme to a suitable form of Borel-Moore homology.
We also consider Lefschetz-Riemann-Roch theorems
when the given group is a split torus, which relate the Riemann-Roch transformations on the given scheme to those on the
fixed points of the torus. These extend various results and formulae previously known, such as the classical results in \cite{AB}, \cite{AS69} as
well as the Riemann-Roch theorems of Gillet in \cite{Gi} and the coherent trace formulae of Thomason in \cite[Theorem 6.4]{Th86-1}.
\vskip .2cm
{\it Our main results will be stated only under the following basic assumptions.} 
\begin{shypothesis}
 \label{stand.hyp.1} 
\begin{enumerate}[\rm(i)] {\rm
\item We will assume the base scheme $\rmS$ is a perfect infinite field $k$ of arbitrary characteristic $p$. 
Let $\rmG$ denote a not-necessarily-connected linear algebraic group defined over $k$.  
\vskip .1cm
\item
We will fix an {\it ambient bigger group}, denoted  $\tilde \rmG$ throughout, and which contains the given linear algebraic group $\rmG$ as a closed sub-group-scheme and 
satisfies the following  strong conditions: $\tilde \rmG$ will denote a connected split reductive group over the field $k$ so that it is {\it special} (in the
 sense of Grothendieck: see \cite{Ch}), and 
 if $\tilde \rmT$ denotes a
maximal torus in $\tilde \rmG$, then $\rmR(\tilde \rmT)$ is free of finite rank over $\rmR(\tilde \rmG)$ and $\rmR(\tilde \rmG)$ is Noetherian. (Here $\rmR(\tilde \rmT)$ and $\rmR(\tilde \rmG)$ denote the
 corresponding representation rings.)
 \vskip .1cm
\item
The above hypothesis is satisfied by 
$\tilde \rmG= {\rm GL}_n$ or ${\rm {SL}}_n$, for any $n$, or any finite product of these groups. It is also trivially satisfied by all split tori.(A basic hypothesis that guarantees this condition is that the algebraic fundamental group 
$\pi_1(\tilde \rmG)$ is 
torsion-free: see \cite[2.3]{CJ23}.) 
\vskip .1cm
\item 
In general, we will restrict to {\it normal} quasi-projective schemes of finite type
over $\rmS$,  and provided with an action by a 
linear algebraic group in \cite[Theorem 1.4]{CJP23}.
\vskip .1cm
However, in order to consider Riemann-Roch transformations with values in suitable equivariant Borel-Moore homology theories,
we need to  restrict to the following class of schemes:
$\rmX$ is a {\it normal} $\rmG$-scheme over $\rmS$ so that it is {\it either} $\rmG$-{\it quasi-projective} (that is, admits a $\rmG$-equivariant locally closed immersion into a projective space over $\rmS$ on which 
$\rmG$-acts linearly: see Definition ~\ref{G.quasiproj}), {\it or} $\rmG$ is connected and $\rmX$ is a normal quasi-projective scheme over $\rmS$ (in which case it is
also $\rmG$-quasi-projective by \cite[Theorem 2.5]{SumII}).}
\end{enumerate}
\end{shypothesis}
\vskip .2cm

\vskip .2cm 
The target of the Riemann-Roch transformations will be, in general, Borel-Moore homology 
theories applied to the Borel-construction. To carry this out in some generality, we start with the framework in \cite{Gi} and then
 expand it a bit to suit the equivariant framework: we summarize this here, and more details may be found in section 2, especially in Definitions ~\ref{G.quasiproj} through 
 Lemma ~\ref{coh.hom.egs}. Accordingly 
 we will assume such cohomology theories are defined by graded complexes
$\Gamma = \oplus_r \Gamma(r)$ of abelian presheaves on one of the following big sites on the base scheme $\rmS= \Speck$, such as 
 either the big Zariski site, the restricted big Zariski site (where the objects are the same as in the big Zariski site, but where
morphisms are required to be flat maps), or the corresponding sites associated to the big \'etale site of $\rmS$. Then the cohomology groups are defined as the corresponding hypercohomology
 with respect to the above complexes, that is, $\rmH^{i, j}(\rmX, \Gamma(*)) = \H^i(\rmX, \Gamma(j))$.  We may also additionally require that
 such bi-graded cohomology theories are contravariantly functorial for all maps between smooth schemes. Moreover, for a 
 fixed linear algebraic group $\rmG$, we will
 restrict to  $\rmG$-quasi-projective schemes.
 \vskip .2cm
The Borel-Moore homology theories we consider will be defined by graded complexes $\Gamma' = \oplus _r \Gamma'(r)$ of abelian presheaves on the same site, and come equipped with pairings: $\Gamma'(r) \otimes \Gamma(s) \ra \Gamma'(r-s)$.
The homology groups are 
defined by $\rmH_{s,t}(\rmX, \Gamma') = \H^{-s}(\rmX, \Gamma'(t))$. We will assume such homology groups are {\it covariantly functorial}
 for proper maps. We require that the above complexes $\{\Gamma (r)|r\}$ and $\{\Gamma'(s)|s\}$
  satisfy a number of properties discussed in detail in Definition ~\ref{gen.hom}. We extend such cohomology (homology) theories
  to the equivariant framework by defining them on a suitable form of the Borel construction for a given linear algebraic group $\rmG$.
 One important additional hypothesis we impose then is the following {\it stability condition}: given a scheme $\rmX$ of finite type over $\k$ and provided with an 
 action by the linear algebraic group $\rmG$, and with $\rmE \tilde \rmG^{gm, m}$ as in \cite[section 3.1]{CJP23}, and for each pair of  fixed
 integers $s, t$, the inverse systems $\{\rmH_{s,t}(\rmE \tilde \rmG^{gm,m}\times _{\rmG}\rmX, \Gamma') |m\}$  and  $\{ \rmH^{s,t}(\rmE \tilde \rmG^{gm,m}\times _{\rmG}\rmX, \Gamma) |m\}$ stabilize as $m \ra \infty$.
 Then we verify in Lemma ~\ref{coh.hom.egs} that a number of cohomology (homology theories) satisfy the above properties. Now we obtain the following
 theorems which are representative of the Riemann-Roch theorems discussed more fully in section 2.
 \vskip .2cm
In view of the difficulties with completions, Equivariant Riemann-Roch theorems such as the following theorem could be 
proven till now only in very restrictive settings, for example, either only at the level of the Grothendieck groups or only for
very restrictive classes of algebraic varieties and schemes (such as projective smooth schemes), or in the non-equivariant framework as in \cite{Gi}. One may observe that
the hypotheses of the theorem are satisfied by all normal quasi-projective schemes  when the group $\rmG$ is assumed to be
connected (in view of Sumihiro's theorem: see \cite[Theorem  2.5]{SumII}), and therefore apply to all toric and spherical varieties.
 
 \begin{theorem}
 \label{composite.RR.}
 Let $\rm X \ra \{H_{s,t}(\rmX, \Gamma')|s,t\}$ denote a homology theory defined in the framework discussed in Definition ~\ref{gen.hom} which extends
 to define an equivariant homology theory satisfying the hypotheses in Definition ~\ref{gen.equiv.coh.hom}.
 \vskip .1cm
 (i) Let $f:\rmX \ra \rmY$ denote
a proper $\rmG$-equivariant map between two normal $\rmG$-quasi-projective schemes. Let $\tilde \rmG$ denote
either a ${\rm GL}_n$ or a finite product of groups of the form ${\rm GL}_n$ containing $\rmG$ as a closed subgroup-scheme. Then one obtains
the commutative diagram, where $\tau_{\rmX}$, $\tau_{\rmY}$ are suitable Riemann-Roch transformations:
\be \begin{equation}
   \label{RR.1.1}
\xymatrix{ {\pi_*{\bf G}(\rmX, \rmG)} \ar@<1ex>[r] \ar@<1ex>[d]^{Rf_*}  & {\limm \pi_*({\bf G}({\rm E}{\tilde \rmG}^{gm,m}{\underset {\rmG} \times}\rmX))} \ar@<1ex>[d]^{Rf_*} \ar@<1ex>[r]^{\tau_{\rmX}} &{\limm \rmH_{*, \bullet}({\rm E}{\tilde \rmG}^{gm,m}{\underset {\rmG} \times}\rmX, \Gamma')_{{\mathbb Q}}} \ar@<1ex>[d]^{f_*}\\
              {\pi_*{\bf G}(\rmY, \rmG)} \ar@<1ex>[r]  &  {\limm\pi_*({\bf G}({\rm E}{\tilde \rmG}^{gm,m}{\underset {\rmG} \times}\rmY))}  \ar@<1ex>[r]^{\tau_{\rmY}} & { \limm\rmH_{*, \bullet}({\rm E}{\tilde \rmG}^{gm,m}{\underset {\rmG} \times}\rmX, \Gamma')_{{\mathbb Q}} }.}
\end{equation} \ee
\vskip .1cm
(ii) Assume in addition to the hypotheses in (i) that the map $f$ is also perfect, for example, a map that is a local complete intersection morphism.
Then one obtains the commutative diagram, where $\tau_{\rmX}$, $\tau_{\rmY}$ are suitable Riemann-Roch transformations:
\be \begin{equation}
   \label{RR.2.1}
\xymatrix{ {\pi_*{\bKH}(\rmX, \rmG)} \ar@<1ex>[r] \ar@<1ex>[d]^{Rf_*}  & {\limm \pi_*({\bKH}({\rm E}{\tilde \rmG}^{gm,m}{\underset {\rmG} \times}\rmX))} \ar@<1ex>[d]^{Rf_*} \ar@<1ex>[r]^{\tau_{\rmX}} &{\limm \rmH_{*, \bullet}({\rm E}{\tilde \rmG}^{gm,m}{\underset {\rmG} \times}\rmX, \Gamma')_{{\mathbb Q}}} \ar@<1ex>[d]^{f_*}\\
              {\pi_*{\bKH}(\rmY, \rmG)} \ar@<1ex>[r]  &  {\limm \pi_*({\bKH}({\rm E}{\tilde \rmG}^{gm,m}{\underset {\rmG} \times}\rmY))}  \ar@<1ex>[r]^{\tau_{\rmY}} & { \limm\rmH_{*, \bullet}({\rm E}{\tilde \rmG}^{gm,m}{\underset {\rmG} \times}\rmX, \Gamma')_{{\mathbb Q}} }.}
\end{equation} \ee
\end{theorem}
\vskip .2cm
Observe that any finite group $\rmG$ may be imbedded in a symmetric group, $\Sigma_n$, for some $n >0$, as a subgroup,
and the symmetric group $\Sigma_n$ imbeds into $\tilde \rmG=\GL_n$ as the closed subgroup of all permutation matrices.  Then we obtain:
\begin{corollary} (Finite group actions)
 \label{finite.group.case} Assume $\rmG$ is a finite group imbedded as a closed subgroup of some $\GL_n$. Let $\tilde \rmG = \GL_n$ and let
  $\rmX$ denote a quasi-projective normal scheme as in ~\ref{stand.hyp.1}(iv) provided with an action by $\rmG$.
 Then Theorems \cite[Theorem 1.4]{CJP23}, and ~\ref{composite.RR.} apply.
 \end{corollary}
\vskip .2cm
The following {\it Lefschetz-Riemann-Roch and the Coherent trace formula} extends a similar result proven by 
 Thomason with the target in his case being Borel-style {\it equivariant topological G-theory}
 with finite coefficients, that is, the Borel style equivariant algebraic G-theory with finite coefficients and with the Bott-element inverted. See \cite[Theorem 6.4]{Th86-1}.
 We are able to lift such theorems to Borel-style equivariant algebraic G-theory and to Borel-style equivariant homotopy K-theory, in general.
\vskip .1cm
\begin{theorem} (Lefschetz-Riemann-Roch and the Coherent trace formula)
 \label{LRR}
  Let $f: \rmX \ra \rmY$ denote a $\rmT$-equivariant proper map between two $\rmT$-schemes of finite type over $\k$ 
  for the action of a split torus $\rmT$. Assume 
  that $f$ factors $\rmT$-equivariantly as $\pi \circ i_1$, where $i_1: \rmX \ra \rmZ$ is
  closed immersion into a regular $\rmT$-scheme $\rmZ$, followed by a proper $\rmT$-equivariant map $\pi: \rmZ \ra \rm Y$. Let $i: \rmZ^{\rmT} \ra \rmZ$ denote
  the closed immersion of the fixed point scheme of $\rmZ$ into $\rmZ$, with $\cN$ denoting the corresponding co-normal sheaf. Let $\pi^{\rmT}: \rmZ^{\rmT} \ra \rmY^{\rmT}$ denote the induced map.
  Let ${\mathcal F}$ denote a $\rmT$-equivariant coherent sheaf on $\rmX$ and let 
  \[{\mathcal G} = Li^*(i_{1*}({\mathcal F})){\underset {\O_{\rmZ^{\rmT}}} \otimes} (\lambda_{-1}(\cN))^{-1}.\]
  \begin{enumerate}[\rm(i)]
  \item Then one has equality of the classes
  \be \begin{equation}
     \label{equality}
  Rf_*({\mathcal F}) = \Sigma_i(-1)^i R^if_*({\mathcal F}) = i_{\rmY*}R\pi^{\rmT}_*({\mathcal G}) = \Sigma_i(-1)^i i_{\rmY*}R^i\pi^{\rmT}_*({\mathcal G})
  \end{equation} \ee
  in $\pi_0(\holimm {\bf G}(\rmE\rmT^{gm,m}\times_{\rmT} \rmY))_{(0)}$ and in $\limm \pi_0({\bf G}(\rmE\rmT^{gm,m}\times_{\rmT} \rmY))_{(0)}$,
  where the subscript ${(0)}$ denotes the localization at the prime ideal $(0)$ in $\rmR(\rmT)$ and $i_{\rmY}: \rmY^{\rmT} \ra \rmY$ denotes the obvious closed immersion.
  \vskip .1cm
  \item The above hypotheses are satisfied whenever $\rmY$ is regular and $\rmX$ is $\rmT$-equivariantly projective, for example,
  $\rmX$ is a normal quasi-projective scheme.
  \vskip .1cm
  \item If $\rmY$ has trivial action by $\rmT$, equality holds in $\pi_0(\holimm {\bf G}(\rmE\rmT^{gm,m}\times_{\rmT} \rmY)) \cong \limm \pi_0({\bf G}(\rmE\rmT^{gm,m}\times_{\rmT} \rmY)) \cong  {\mathbb Z}[\rmT]\compl_{\rmI_{\rmT}} \otimes \pi_0(\bG(\rmY))$.
  \vskip .1cm
  \item  If $\rmY$ is affine and $\rmX^{\rmT}$ is affine (for example, $\rmX^{\rmT}$ is a finite set of isolated points on $\rmX$), then
  one obtains $\Sigma_i (-1)^i \rmH^i(\rmX, {\mathcal F}) = \rmH^0(\rmX^{\rmT}, {\mathcal G})$. This hypothesis is satisfied if
  $\rmX$ is a projective toric variety for a split torus $\rmT$, or a projective spherical variety for a linear algebraic group $\rmG$ over an algebraically closed
  field, with $\rmT$ denoting the maximal torus in $ \rmG$ and where $\rmY$ denotes the base field.
  \vskip .1cm
  \item If the map $\pi$ is perfect and ${\mathcal F}$ denotes a perfect complex of $\rmT$-equivariant $\O_{\rmX}$-modules,
  then the results corresponding to (i) through (iv) hold for equivariant homotopy K-theory replacing equivariant G-theory.
  \end{enumerate}
  \end{theorem}
  \begin{corollary}
  \label{LRR.cor}
  \begin{enumerate}[\rm(i)]
  \item Assume that $\rmX$ and $\rmY$ are $\rmT$-quasi-projective and ${\mathcal F}$ is as in Theorem ~\ref{LRR}(i). Composing with the Riemann-Roch transformations into a suitable equivariant Borel-Moore homology theory (see Definition 
  ~\ref{gen.equiv.coh.hom})  localized at the prime ideal
  $(0)$ in $\limm \Pi_{\rm r,s} \rmH^{r,s}(\rmB\rmT^{gm,m}, \Gamma)_{\Q}$, one also obtains 
  \[f_*(\tau_{\rmZ}({\mathcal F})) = \tau_{\rmY}(R\pi_*({\mathcal F})) = i_{\rmY*}\pi^{\rmT}_*\tau_{\rmZ^{\rmT}}({\mathcal G}) =\pi_*\tau_{\rmZ}(i_*{\mathcal G}).\]
  as classes in $(\limm \Pi_{\rm u,v}\rmH_{u, v}(\rmE\rmT^{gm,m}\times_{\rmT} \rmY, \Gamma'))_{(0)}$. (Here $\tau_{\rmZ}$ ($\tau_{\rmY}$, $\tau_{\rmZ^{\rmT}}$) denote the
   Riemann-Roch transformation associated to $\rmZ$( $\rmY$, $\rmZ^{\rmT}$ \res) and a generalized Borel-Moore homology theory.)
 \item
 If the map $\pi$ is perfect and ${\mathcal F}$ denotes a perfect complex of $\rmT$-equivariant $\O_{\rmX}$-modules,
  then the results corresponding to (i)  holds for equivariant homotopy K-theory replacing equivariant G-theory.
 \item  (Todd class formula)
  Under the assumptions of (i),
  \[ f_*(\tau_{\rmX}) = \pi_*\tau_{\rmZ}(i_*(Li^*(i_{1*}({\O_{\rmX}})){\underset {\O_{\rmZ^{\rmT}}} \otimes} (\lambda_{-1}(\cN))^{-1})).\]
   \end{enumerate}
\end{corollary}
\begin{corollary} 
\label{Dem.toric}
  (Demazure character formula): Let $\rmG/\rmB$ denote the flag variety associated to a connected reductive group $\rmG$ and a Borel subgroup $\rmB$ all defined over an algebraically closed field $k$.
  Let  $\rmX_{\it w}$ denote a Schubert variety in $\rmG/\rmB$, associated to $ {\it w } \in {\rm W}$, which denotes the corresponding Weyl group. Let $\rmT$ denote 
  the maximal torus in $\rmG$. 
  Let ${\mathcal L}_{\lambda}$ denote the $\rmT$-equivariant line bundle on $\rmG/\rmB$ associated to a character $\lambda $ of $\rmT$.
  Then in $\rmR(\rmT) = \pi_0(\bG(\Speck, \rmT)$, one obtains:
  \[\Sigma_i(-1)^i\rmH^i(\rmX_{\it w}, {\mathcal L}_{\lambda}) = e^{\rho}L_{{\it s}_1} \cdots L_{\it s_r}e^{\lambda- \rho}\]
  where ${\it w} = s_1\cdots s_r$, ${\rm L}_{{\it s}_{\alpha}}(u) = \frac{u-s_{\alpha}u}{1-e^{\alpha}}$, and $\rho$ is half the sum of the positive roots.
\end{corollary}

\vskip .2cm
As another application of the derived completion theorems in \cite{CJ23}) (see \cite[theorems 1.2 and 5.6]{CJ23}), we obtain the following result that shows how to compute the homotopy groups
of the equivariant G-theory in terms of equivariant motivic cohomology.   Assume that $\rmX$ is a $ \rmG$-quasi-projective scheme, as in ~\ref{G.quasiproj}, with $\rmG= \rmT$ a split torus. Then the approximation $\rmE\tilde \rmG^{gm,m}{\underset {\tilde \rmG} \times}\rmX$ (for any integer $m>0$), is always a 
scheme, and we obtain the following theorem.
\begin{theorem}
 \label{compute.der.compl}
  Assume that $\rmX$ is a $ \rmT$-quasi-projective scheme as in Definition ~\ref{G.quasiproj}, with $ \rmT$ a split torus. Then, for a fixed value of $-s-t$, and for $m$ sufficiently large, there exists a strongly convergent spectral sequence with:
  \[E_2^{s,t} = \H_{\rm BM,M}^{s-t}(\rmE \rmT^{gm,m}{\underset { \rmT} \times}\rmX, {\mathbb Z}(-t)) \Ra \pi_{-s-t}(\bG(\rmE \rmT^{gm,m}{\underset { \rmT} \times}\rmX)) \cong \pi_{-s-t}(\bG(\rmX,  \rmT)\compl_{\rho_{ \rmT, m}}), \]
 where $\H_{\rm BM,M}$ denotes Borel-Moore motivic cohomology, which identifies with the higher equivariant Chow groups. 
 \vskip .1cm
 In addition, rationally, the homotopy groups 
$\pi_{-s-t}(\bG(\rmX,  \rmT)\compl_{\rho_{\rmT, m}})$ identify with the rational $ \rmT$-equivariant higher Chow  groups. 
 \end{theorem}
\vskip .1cm
\begin{remark} For more general groups other than a split torus, there is still
  a similar spectral sequence, but only for the full $\rmE \rmG^{gm}{\underset { \rmG} \times}\rmX$. The corresponding spectral sequence converges only conditionally, and computes the homotopy groups of the full derived completion.
\end{remark}
\vskip .2cm
We conclude this discussion by providing the following corollaries to clarify the range of applications of the above Theorems.
\vskip .1cm
The results of Sumihiro (see \cite[Theorem 1]{SumI} and \cite[Theorem 2.5]{SumII}) and the following Corollary  enable us to see that the Theorems \cite[Theorem 1.4]{CJP23}, Theorems ~\ref{composite.RR.}, ~\ref{LRR} and ~\ref{compute.der.compl} all apply
to toric varieties and spherical varieties, irrespective of whether they are regular or projective. Recall a {\it normal} variety $\rmX$ provided
 with the action of a split torus $\rmT$ (of a connected split reductive group $ \rmG$) is called {\it a toric variety} ({\it an $ \rmG$-spherical variety}) if $\rmT$ has only finitely many orbits 
  on $\rmX$ ($ \rmG$ and a Borel subgroup of $ \rmG$ both have only finitely many orbits on $\rmX$, \res).
\begin{corollary} (Toric varieties and Spherical varieties)
 \label{toric.sph}
 The Theorems \cite[Theorem 1.4]{CJP23}, ~\ref{composite.RR.}, ~\ref{LRR}, ~\ref{compute.der.compl} and Corollary ~\ref{LRR.cor} apply to all quasi-projective split toric varieties, with the group $ \rmG$ denoting the given torus. Theorems \cite[Theorem 1.4]{CJP23} and ~\ref{composite.RR.} 
 apply to all quasi-projective $ \rmG$-spherical varieties, where $ \rmG$ is the given connected reductive group. Theorems ~\ref{LRR} and ~\ref{compute.der.compl} apply to 
 all quasi-projective $ \rmG$-spherical varieties for the action of a split sub-torus of the given group $ \rmG$.
\end{corollary}

\vskip .2cm

\vskip .2cm
Here is an {\it outline of the paper}.    Section 2 discusses  various forms of 
equivariant Riemann-Roch theorems culminating in those taking values in any equivariant Borel-Moore motivic cohomology theory. Section 3 then discusses
the Lefschetz-Riemann-Roch theorem involving the action of a torus and also discuss various applications stemming from this. 
Section 4 discusses how to compute the completed equivariant G-theory in terms of
equivariant motivic cohomology. 

\section{Higher Equivariant Riemann-Roch theorems}
We consider several variants of equivariant Riemann-Roch theorems in this section. 
\subsection{Riemann-Roch from forms of equivariant K-theory to the corresponding form of K-theory on the Borel construction}
The first theorem considers the Riemann-Roch 
transformation between equivariant $\bG$-theory and Borel-style equivariant $\bG$-theory and the second theorem considers 
the Riemann-Roch 
transformation between equivariant Homotopy K-theory and Borel-style equivariant homotopy K-theory.
\begin{theorem}(Riemann-Roch for Equivariant G-theory)
\label{RR.Gth}
 Let $f:\rmX \ra \rmY$ denote
a proper $\rmG$-equivariant map  between two normal quasi-projective $\rmG$-schemes of finite type over $\rmS= \Speck$. Let $\tilde \rmG$ denote
either a ${\rm GL}_n$ or a finite product of groups of the form ${\rm GL}_n$ containing $\rmG$ as a closed subgroup-scheme. Then 
the squares
\be \begin{equation}
\label{RR1.1}
\xymatrix{ {{\bf G}(\rmX, \rmG)} \ar@<1ex>[r] \ar@<1ex>[d]^{Rf_*} & {{\bf G}(\rmX, \rmG)\compl_{\rho_{{\tilde \rmG}, \alpha(m)}}} \ar@<1ex>[r] \ar@<1ex>[d]^{Rf_*} & {{\bf G}({\rm E}{ \tilde \rmG}^{gm,m}\times \rmX, \rmG)} \ar@<1ex>[d]^{Rf_*} \ar@<1ex>[r]^{\simeq} & {{\bf G}({\rm E}{\tilde \rmG}^{gm,m}{\underset {\rmG} \times}\rmX)} \ar@<1ex>[d]^{Rf_*}\\
              {{\bf G}(\rmY, \rmG)} \ar@<1ex>[r]  & {{\bf G}(\rmY, \rmG)\compl_{\rho_{{ \rmG}, \alpha(m)}}} \ar@<1ex>[r]  &  {{\bf G}({\rm E}{\tilde \rmG}^{gm,m}\times \rmY, \rmG)} \ar@<1ex>[r]^{\simeq} & {{\bf G}({\rm E}{\tilde \rmG}^{gm,m}{\underset {\rmG} \times}\rmY)}}
\end{equation} \ee
homotopy commute, for any fixed integer $m \ge 0$, with $\alpha(m)$ as in \cite[Definition 5.5]{CJP23}. Moreover, as $m$ varies 
one obtains an inverse system of such homotopy commutative diagrams. On taking the homotopy inverse limits we obtain the  commutative diagram:
\be \begin{equation}
\label{RR1.2}
\xymatrix{ {\pi_*{\bf G}(\rmX, \rmG)} \ar@<1ex>[r] \ar@<1ex>[d]^{Rf_*} & {\pi_*(\holimm {\bf G}(\rmX, \rmG)\compl_{\rho_{{\tilde \rmG}, \alpha(m)}})} \ar@<1ex>[r]^{\cong} \ar@<1ex>[d]^{Rf_*} & {\pi_*(\holimm {\bf G}({\rm E}{ \tilde \rmG}^{gm,m}\times \rmX, \rmG))} \ar@<1ex>[d]^{Rf_*} \ar@<1ex>[r]^{\cong} & {\pi_*(\holimm {\bf G}({\rm E}{\tilde \rmG}^{gm,m}{\underset {\rmG} \times}\rmX))} \ar@<1ex>[d]^{Rf_*}\\
              {\pi_*{\bf G}(\rmY, \rmG)} \ar@<1ex>[r]  & {\pi_*(\holimm {\bf G}(\rmY, \rmG)\compl_{\rho_{{ \rmG}, \alpha(m)}})} \ar@<1ex>[r]^{\cong}  &  {\pi_*(\holimm {\bf G}({\rm E}{\tilde \rmG}^{gm,m}\times \rmY, \rmG))} \ar@<1ex>[r]^{\cong} & {\pi_*(\holimm {\bf G}({\rm E}{\tilde \rmG}^{gm,m}{\underset {\rmG} \times}\rmY)}).}
\end{equation} \ee
\end{theorem}
\begin{theorem}(Riemann-Roch for Equivariant Homotopy K-theory)
\label{RR.KH}
Let $f:\rmX \ra \rmY$ denote
a proper $\rmG$-equivariant map between two normal quasi-projective $\rmG$-schemes of finite type over $\rmS= \Speck$, that is also {\it perfect} (see \cite[Definition 2.1]{CJP23}). Let $\tilde \rmG$ denote
either a ${\rm GL}_n$ or a finite product of groups of the form ${\rm GL}_n$ containing $\rmG$ as a closed subgroup-scheme. Then 
the squares
\be \begin{equation}
\label{RR2.1}
\xymatrix{ {{\bKH}(\rmX, \rmG)} \ar@<1ex>[r] \ar@<1ex>[d]^{Rf_*} & {{\bKH}(\rmX, \rmG)\compl_{\rho_{{\tilde \rmG}, \alpha(m)}}} \ar@<1ex>[r] \ar@<1ex>[d]^{Rf_*} & {{\bKH}({\rm E}{\tilde \rmG}^{gm,m}\times \rmX, \rmG)} \ar@<1ex>[d]^{Rf_*} \ar@<1ex>[r]^{\simeq} & {{\bKH}({\rm E}{\tilde \rmG}^{gm,m}{\underset {\rmG} \times}\rmX)} \ar@<1ex>[d]^{Rf_*}\\
              {{\bKH}(\rmY, \rmG)} \ar@<1ex>[r]  & {{\bKH}(\rmY, \rmG)\compl_{\rho_{{\tilde \rmG}, \alpha(m)}}} \ar@<1ex>[r]  &  {{\bKH}({\rm E}{\tilde \rmG}^{gm,m}\times \rmY, \rmG)} \ar@<1ex>[r]^{\simeq} & {{\bKH}({\rm E}{\tilde \rmG}^{gm,m}{\underset {\rmG} \times}\rmY)}}
\end{equation} \ee
homotopy commute, for any fixed integer $m \ge 0$, with $\alpha(m)$ as in \cite[Definition 5.5]{CJP23}.  Moreover, as $m$ varies 
one obtains an inverse system of such homotopy commutative diagrams. On taking the homotopy inverse limits we obtain the commutative diagram:
\fontsize{9}{12}
\be \begin{equation}
\label{RR2.2}
\xymatrix{ {\pi_*{\bKH}(\rmX, \rmG)} \ar@<1ex>[r] \ar@<1ex>[d]^{Rf_*} & {\pi_*(\holimm {\bKH}(\rmX, \rmG)\compl_{\rho_{{\tilde \rmG}, \alpha(m)}})} \ar@<1ex>[r]^{\cong} \ar@<1ex>[d]^{Rf_*} & {\pi_*(\holimm {\bKH}({\rm E}{ \tilde \rmG}^{gm,m}\times \rmX, \rmG))} \ar@<1ex>[d]^{Rf_*} \ar@<1ex>[r]^{\cong} & {\pi_*(\holimm {\bKH}({\rm E}{\tilde \rmG}^{gm,m}{\underset {\rmG} \times}\rmX))} \ar@<1ex>[d]^{Rf_*}\\
              {\pi_*{\bKH}(\rmY, \rmG)} \ar@<1ex>[r]  & {\pi_*(\holimm {\bKH}(\rmY, \rmG)\compl_{\rho_{{ \rmG}, \alpha(m)}})} \ar@<1ex>[r]^{\cong} &  {\pi_*(\holimm {\bKH}({\rm E}{\tilde \rmG}^{gm,m}\times \rmY, \rmG))} \ar@<1ex>[r]^{\cong} & {\pi_*(\holimm {\bKH}({\rm E}{\tilde \rmG}^{gm,m}{\underset {\rmG} \times}\rmY))} .}
\end{equation} \ee  
\normalsize
\end{theorem}
\vskip .1cm \noindent
{\bf Proofs of the above two theorems.} 
The commutativity of the left-most square in both ~\eqref{RR1.1} and ~\eqref{RR2.1} follows from 
the definition and functoriality of derived completion for equivariant G-theory and equivariant homotopy K-theory as in \cite[3.0.7]{CJ23}: see also \cite[(4.0.3)]{CJP23}. The composition of the left two horizontal maps in the top row of 
~\eqref{RR1.1} (\eqref{RR2.1}) is given by sending a complex of $\rmG$-equivariant
coherent sheaves on $\rmX$ ($\rmG$-equivariant perfect complex on $\rmX \times \Delta[n]$) to its pull-back on ${\rm E}{\tilde \rmG}^{gm,m}{ \times}\rmX$ 
(${\rm E}{\tilde \rmG}^{gm,m}{ \times}\rmX \times \Delta [n]$, \res). The composition of the left two horizontal maps in the bottom row of ~\eqref{RR1.1} (\eqref{RR2.1})
has a similar description. 
\vskip .1cm
Therefore, the homotopy commutativity of the of the  middle  square in ~\eqref{RR1.1} (\eqref{RR2.1}) follows from 
\cite[Proposition 4.5(ii)]{CJ23} (\cite[Proposition 5.4(ii)]{CJP23}, \res).
The commutativity of the last squares in both ~\eqref{RR1.1} and ~\eqref{RR2.1} is clear. In order to show that the homotopy commutative diagrams ~\eqref{RR1.1} and ~\eqref{RR2.1} in both 
theorems are compatible as $m$ varies, resulting in the commutative diagrams in ~\eqref{RR1.2} and ~\eqref{RR2.2}, we invoke \cite[Proposition 5.7]{CJP23}. 
\vskip .1cm 
\qed

\subsection{Equivariant cohomology and equivariant Borel-Moore homology theories}
\vskip .2cm

We proceed to combine the above Riemann-Roch theorems with any of the usual Riemann-Roch theorems into various cohomology/homology theories.
For this we will adopt the framework as in \cite{Gi}. We start with the base scheme $\rmS= \Speck$, the spectrum of
a field $k$, and fix either the big Zariski site, the restricted big Zariski site (where the objects are the same as in the big Zariski site, but where
morphisms are required to be flat maps), the full subcategory of the big Zariski site consisting of all smooth schemes of finite type over $\rmS$, 
or the corresponding sites associated to the big \'etale site of $\rmS$.
\begin{definition}
 \label{gen.coh}
Let $\Gamma = \oplus_r \Gamma(r)$ denote a graded complex of abelian presheaves on one of the above sites. Then the cohomology groups are defined as the corresponding hypercohomology
 with respect to the above complexes, that is, $\rmH^{i, j}(\rmX, \Gamma(*)) = \H^i(\rmX, \Gamma(j))$. (The index $i$ ($j$) will be called the degree (the weight, \res).) We may also additionally require that
 such bi-graded cohomology theories are contravariantly functorial for all maps between smooth schemes.
\end{definition} 
 \vskip .2cm
 In order to be able to define a suitable Riemann-Roch transformation with values in an equivariant Borel-Moore homology theory, we will
need to restrict to a nicer class of schemes as considered in the following definition.
\begin{definition}
\label{G.quasiproj}
 For a given linear algebraic group $\rmG$, a convenient class of schemes we consider
 are what are called {\it $\rmG$-quasi-projective schemes}: a quasi-projective scheme $\rmX$ with a given $\rmG$-action is
 $\rmG$-quasi-projective, if it admits a $\rmG$-equivariant locally closed immersion into a projective space,  ${\rm Proj(V)}$,
 for a representation ${\rm V}$ of $\rmG$. Then a well-known theorem of Sumihiro, (see \cite[Theorem 2.5]{SumII}) shows that any normal quasi-projective scheme over $\rmS$ provided
 with the action of a connected linear algebraic group $\rmG$ is $\rmG$-quasi-projective. This may be extended to the 
 case where $\rmG$ is not necessarily connected by invoking the above result for the action of $\rmG^o$ and then
 considering the closed $\rmG$-equivariant closed immersion of the given $\rmG$-scheme into 
 $\Pi_{g \eps \rmG/\rmG^o} {\rm Proj(V)} $ followed by a $\rmG$-equivariant closed immersion of the latter into
 ${\rm Proj}(\otimes_{g \eps \rmG/\rmG^o}{\rm V})$. (See for example. \cite[pp. 629-630]{Th88}.)
\end{definition}

 \begin{definition} (Generalized Borel-Moore homology theories)
  \label{gen.hom}
  The following assumptions are essentially those that appear in \cite{Gi}.
The homology theories we consider will be defined by graded complexes $\Gamma' = \oplus _r \Gamma'(r)$, that come equipped with pairings: $\Gamma'(r) \otimes \Gamma(s) \ra \Gamma'(r-s)$, with the homology groups
defined by $\rmH_{s,t}(\rmX, \Gamma') = \H^{-s}(\rmX, \Gamma'(t))$. We will assume such homology groups are {\it covariantly functorial}
 for proper maps, with $f_*:\rmH_{s, t}(\rmX, \Gamma') \ra \rmH_{s, t}(\rmY, \Gamma')$ denoting the induced map associated 
 to such a proper map $\rmX \ra \rmY$. We will let $\rmH_{*, \bullet}(\rm X, \Gamma')$ denote  $ \oplus_{s, t}\rmH_{s, t}(\rmX, \Gamma')$. Moreover, $\rmH_{*, \bullet}(\rmX, \Gamma')_{{\mathbb Q}}$ will
 denote the corresponding graded homology group, that is made into a graded ${\mathbb Q}$-vector space in a suitable manner, depending on the choice of the 
 complexes as in the examples considered below.
 We will also impose the following additional requirements:
 \begin{enumerate}[\rm(i)]
  \item If $f: \rmX \ra \rmY$ is a flat map of pure relative dimension $d$, then there is an induced map $f^*: \rmH_{s,t}(\rmY, \Gamma') \ra \rmH_{s+2d, t+d}(\rmX, \Gamma')$.
  \item For a closed immersion $i: \rmY \ra \rmX$, there exists a long-exact localization sequence:
  \[\cdots \ra \rmH_{i,j}(\rmY, \Gamma') {\overset {i_*} \ra} \rmH_{i,j}(\rmX, \Gamma') {\overset {j^*} \ra} \rmH_{i, j}(\rmX - \rmY, \Gamma') \ra \rmH_{i-1, j}(\rmY, \Gamma') \ra \cdots .\]
 \item The given pairing $\Gamma'(r) \otimes \Gamma(s) \ra \Gamma'(r-s)$ induces a (cap-product) pairing for every closed subscheme $\rmY$ of $\rmX$:
 \[\rmH_{i,j}(\rmX, \Gamma') \otimes \rmH_{\rmY}^{r, s}(\rmX, \Gamma) \ra \rmH_{i-r, j-s}(\rmY, \Gamma').\]
 \item For a smooth scheme $\rmX$ of pure dimension $d$, there exists a fundamental class $[\rmX] \eps H_{2d,d}(\rmX, \Gamma')$ so that the cap-product
 pairing with this class induces an isomorphism: $\rmH_{\rmY}^{r, s}(\rmX, \Gamma) \ra \rmH_{2d-r, d-s}(\rmY, \Gamma')$, for $\rmY$ closed and of pure codimension in $\rmX$.
\item For a regular closed immersion $i: \rmY \ra \rmX$ of smooth schemes of pure codimension $c$, one obtains Thom-isomorphism, that is, an isomorphism: 
 $\rmH^{i,j}(\rmY, \Gamma) {\overset {\cong} \ra} \rmH^{i+2c, j+c}_{\rmY}(\rmX, \Gamma)$.
  \item (Riemann-Roch transformation) Given such a cohomology/homology theory, we also require that one has a Riemann-Roch transformation
 $\tau: \pi_*({\bf G}(\rmX)) \ra \rmH_{*, \bullet}(\rmX, \Gamma')_{{\mathbb Q}}$ defined for all the schemes considered and which is covariantly functorial 
 in $\rmX$ for proper maps.
 \item Gysin maps for regular closed immersions: if $i: \rmY \ra \rmX$ is a regular closed immersion of pure codimension $c$, we assume there is a Gysin map:
 $i^*: \rmH_{s, t}(\rmX, \Gamma') \ra \rmH_{s-2c, t-c}(\rmY, \Gamma')$.
 \end{enumerate}
 \end{definition}
 Let $\rmX$ denote a $\rmG$-quasi-projective scheme, together with a $\rmG$-equivariant closed immersion into a smooth $\rmG$-scheme
 $\tilde \rmX$. Then we define a Gysin map $i^*: \rmH_{*, \bullet}(\rmE \tilde \rmG^{gm,m+1}\times _{\rmG} \rmX, \Gamma') \ra \rmH_{*-2c, \bullet-c}(\rmE \tilde \rmG^{gm,m}\times _{\rmG} \rmX, \Gamma')$
 by making use of the vertical maps in the following diagram which are isomorphisms:
 \be \begin{equation}
 \label{Gysin.eq.homology}
 \xymatrix{{\rmH_{*, \bullet}(\rmE \tilde\rmG^{gm,m+1}\times _{\rmG} \rmX, \Gamma')} \ar@<1ex>[r]  & {\rmH_{*-2c, \bullet-c}(\rmE \tilde \rmG^{gm,m}\times _{\rmG} \rmX, \Gamma')} \\
             {\rmH^{2\tilde d-*, \tilde d-\bullet}_{\rmE \tilde \rmG^{gm,m+1}\times _{\rmG} \rmX}(\rmE \tilde \rmG^{gm,m+1}\times _{\rmG} \tilde \rmX, \Gamma)} \ar@<1ex>[u]^{\cong} \ar@<1ex>[r]^{i^*} & {\rmH^{2\tilde d-*, \tilde d-\bullet}_{\rmE \tilde \rmG^{gm,m}\times _{\rmG} \rmX}(\rmE \tilde \rmG^{gm,m}\times _{\rmG} \tilde \rmX, \Gamma)}\ar@<1ex>[u]^{\cong} }
 \end{equation} \ee
where the map in the bottom row is the obvious restriction, $\tilde d = dim(\rmE \tilde \rmG^{gm,m+1} \times_{\rmG} \tilde \rmX)$
and $c=  dim(\rmE \tilde \rmG^{gm,m+1} \times_{\rmG} \tilde \rmX) - dim(\rmE \tilde \rmG^{gm,m} \times_{\rmG} \tilde \rmX)$.
In addition to the above hypotheses, we will require the following additional ones to handle the equivariant framework.
\vskip .2cm
{\it Stability for equivariant cohomology and homology}.
\label{stab.1}
\begin{enumerate}[\rm(i)]
 \item Let $\tilde \rmG$ denote
either a ${\rm GL}_n$ or a finite product of groups of the form ${\rm GL}_n$ containing the given linear algebraic group $\rmG$ as a closed subgroup-scheme.
 We will often make the {\it following additional assumption}: given a scheme $\rmX$ of finite type over $\k$ and provided with an 
 action by the linear algebraic group $\rmG$, and with $\rmE \tilde \rmG^{gm, m}$ as in \cite[section 3.1]{CJP23}, and for each pair of  fixed
 integers $\rmi, \rmj$, the inverse system $\{\rmH^{i,j}(\rmE \tilde \rmG^{gm,m}\times _{\rmG}\rmX, \Gamma) |m\}$ stabilizes as $m \ra \infty$.
 \item Assume the dimension of $\rmE \tilde \rmG^{gm, m}$ is ${\tilde {\rm d}}_m$ and the dimension of $\rmG$ is ${\rm g}$. Then
 we also require that for each fixed $\rmi, \rmj$, the inverse system $\{\rmH_{ i+2\tilde d_m-2g, j+\tilde d_m -g}(\rmE \tilde \rmG^{gm,m}\times _{\rmG}\rmX, \Gamma')|m\}$
 (where the structure maps of the inverse system are defined by the Gysin maps in ~\eqref{Gysin.eq.homology}), stabilizes as $m \ra \infty$.
 \end{enumerate}
 \begin{definition} (Generalized equivariant cohomology/homology theories)
  \label{gen.equiv.coh.hom}
  \begin{enumerate}[\rm(i)]
 \item We will let $\rmH_{\rmG}^{ i,j}(\rmX, \Gamma)$) denote the limiting value of the inverse system $\{\rmH^{i,j}(\rmE \tilde \rmG^{gm,m}\times _{\rmG}\rmX, \Gamma) |m\}$ stabilizes as $m \ra \infty$.
 \item We will let $\rmH^{\rmG}_{i,j}(\rmX, \Gamma)$) denote the corresponding limiting value of the inverse system \newline \noindent 
 $\{\rmH_{ i+2\tilde d_m-2g, j+\tilde d_m -g}(\rmE \tilde \rmG^{gm,m}\times _{\rmG}\rmX, \Gamma')|m\}$ as $m \ra \infty$.
  \item An additional assumption we impose is that the equivariant cohomology $\rmH_{\rmG}^{i, j}(\rmX, \Gamma)$
 and equivariant homology $\rmH^{\rmG}_{s, t}(\rmX, \Gamma')$ for each fixed $i, j, s, t$ are independent of the choice of the ind-scheme 
 $\{\rmE\tilde \rmG^{gm,m}\times_{\rmG}\rmX|m\}$ defining the Borel construction.
 \end{enumerate}
 \end{definition}
 
 \begin{definition}
  \label{equiv.fund.class} (Equivariant fundamental class) Assume $\rmX$ is a $\rmG$-quasi-projective scheme of pure dimension $d$.
  In view of the stability of the inverse system 
  $\{\rmH_{2d+2\tilde d_m-2g, d+\tilde d_m-g}(\rmE \tilde \rmG^{gm,m}\times _{\rmG} \rmX, \Gamma')|m\}$, 
  we define $\rmH^{\rmG}_{2d,d}(\rmX, \Gamma') = \rmH_{2d+2\tilde d_m-2g, d+\tilde d_m -g}(\rmE \tilde \rmG^{gm,m}\times _{\rmG} \rmX, \Gamma')$, for $m$ sufficiently large,
  and define the corresponding {\it equivariant fundamental class} $[ \rmX]_{\rmG} \in \rmH^{\rmG}_{2d,d}(\rmX, \Gamma')$
  as the image of the classes 
  \[[ \rmX]_{\rmG, m} \in \rmH_{2d+2\tilde d_m-2g, d+\tilde d_m-g}(\rmE \tilde \rmG^{gm,m}\times _{\rmG} \rmX, \Gamma')\]
  for $m$ sufficiently large.
\end{definition}

\begin{examples} Next we list some well-known examples of cohomology/homology theories satisfying the hypotheses in Definitions ~\ref{gen.coh} and ~\ref{gen.hom}.
 \begin{enumerate}[\rm(i)]
  \item One takes the big Zariski or etale site of all smooth schemes of finite type over $\rmS$ and let $\Gamma (r)$ denote the motivic complex of weight $r$.
  \item One may take the category of all quasi-projective schemes of finite type over $\rmS$ provided with the restricted Zariski topology. Here ones
  may take $\Gamma'(r) $ to be Bloch's higher cycle complex of weight $r$: ${\rm z}^r(\quad , \bullet)$, re-indexed using the dimension of cycles rather
  than the codimension, that is, if $\rmX$ has pure dimension $d$, $({\rm z}_{d-r}(\rmX, \bullet))_n$ will denote the cycles of dimension $d-r+n$ in $\rmX \times \Delta[n]$ that intersect the 
  faces of $\rmX \times \Delta [n]$ properly. The corresponding Zariski hypercohomology groups
   identify with Bloch's higher Chow groups, re-indexed by dimension. These are covariantly functorial for proper maps and have localization sequences. Moreover restricting to 
   smooth schemes, these identify with the motivic cohomology groups, indexed by the codimension of cycles. 
   \item Deligne cohomology and homology: these are defined for schemes of finite type over the field of complex or real numbers by complexes satisfying most
   of the above properties. (See \cite{Jan}.)
   \item \'Etale cohomology with $\Gamma(r) = \mu_{\ell}(r)$, $\ell $ invertible in the base field $k$. The corresponding (Borel-Moore) homology theory is defined
   by $\Gamma '(r)$, the corresponding dualizing complex. One may also consider $\ell$-adic \'etale cohomology and homology.
   \item Singular cohomology for varieties over $k= {\mathbb C}$ or ${\mathbb R}$, with $\Gamma (r) = {\mathbb Z}^{\otimes r}$, where ${\mathbb Z}$ denotes
    the constant sheaf given by the integers. The corresponding homology theory is the familiar Borel-Moore homology theory.
 \end{enumerate}
 \end{examples}
 We will skip the verification that most of the above cohomology/homology theories satisfy the properties listed in Definitions ~\ref{gen.coh} and ~\ref{gen.hom}. However, it is
 important for us to show that the additional hypotheses in Definition ~\ref{gen.equiv.coh.hom} are indeed satisfied by several of the 
 cohomology/homology theories listed above and extended to the equivariant framework. One may extend any cohomology/homology theory satisfying the 
 hypotheses in Definitions ~\ref{gen.coh} and ~\ref{gen.hom} to the equivariant framework, that is, to schemes $\rmY$ that come equipped with the action of a linear algebraic group $\rmG$, by 
 taking the hypercohomology with respect to the complexes $\Gamma(r)$ and $\Gamma'(s)$
 on the Borel construction $\rmE \rmG^{gm,m}\times _{\rmG}\rmY$.
 \begin{lemma}
 \label{coh.hom.egs}
 The following cohomology/homology theories extend to the equivariant framework defining cohomology/homology theories satisfying the hypotheses in Definition ~\ref{gen.equiv.coh.hom}.
  \begin{enumerate}[\rm(i)]
   \item Bloch's higher Chow groups 
   \item $\ell$-adic \'etale cohomology and (Borel-Moore) homology, where $\ell$ denotes a prime that is invertible in the base field $\k$ and $\k$ has finite $\ell$ cohomological dimension.
   \item Singular cohomology and (Borel-Moore) homology, when the base field is the complex numbers.
  \end{enumerate}
\end{lemma}
\begin{proof}
The extension of Bloch's higher Chow groups to the equivariant framework (called equivariant higher Chow groups) is carried out in \cite{Tot} and \cite{EG}.
The properties in Definition ~\ref{gen.equiv.coh.hom} are established in \cite{EG}. 
\vskip .1cm
We will next consider the case of $\ell$-adic \'etale cohomology. Let $\rmBG$, $\rmE \rmG{\underset {\rmG} \times} \rmX$ 
denote the simplicial scheme defined by the simplicial Borel construction as in \cite{J93} and \cite[section 5]{J20}. 
(As should be clear from the discussion in \cite[section 5]{J20}, one could replace the simplicial Borel construction 
$\rmE \rmG{\underset {\rmG} \times} \rmX$ by the ind-scheme $\{\rmE  \rmG^{gm,m}{\underset {\rmG} \times} \rmX|m\}$
or the ind-scheme $\{\rmE \tilde \rmG^{gm,m} {\underset {\rmG} \times} \rmX|m\}$ for a larger linear algebraic group $\tilde \rmG$ containing $\rmG$ as a
 closed subgroup-scheme, 
 as they will give equivalent derived categories. Therefore, we will use the simplicial Borel construction $\rmE\rmG{\underset {\rmG} \times} \rmX$ throughout the following discussion.)
Let $\pi_{\rmX}: \rmE \rmG{\underset {\rmG} \times} \rmX \ra \rmBG$
denote the obvious map induced by the projection $\rmX \ra {\rm Spec} \k$. Then we let $\bDX = R\pi_X^!(\Q_{\ell})$ denote the dualizing complex
on $\rmE \rmG{\underset {\rmG} \times} \rmX$. We will define 
\be \begin{equation}
\label{equiv.et.hom}
\rmH_{s,t}^{\rmG}(\rmX, \Q_{\ell}) = \H_{s,t}(\rmE \rmG{\underset {\rmG} \times} \rmX, \Q_{\ell}) = \H^{-s}(\rmE \rmG{\underset {\rmG} \times} \rmX, \bDX(-t)).
\end{equation} \ee
\vskip .1cm
Next we consider the spectral sequence, for any complex of $\ell$-adic sheaves $\rmK$ on $\rmE \rmG{\underset {\rmG} \times} \rmX$:
\be \begin{equation}
 \label{equiv.coh.ss}
\rmE_2^{\it u,v} = \rmH^{\it u}(\rmBG, R^{\it v}\pi_{X*}(\rmK)) \Ra \H^{\it u+v}(\rmE \rmG{\underset {\rmG} \times} \rmX, \rmK).
\end{equation} \ee
Taking $\rmK = \bDX$ in the above spectral sequence, 
one may observe that there is a fundamental class $[\rmX_{\bar {\it x}}] \in R^{-2d}\pi_{X*}(\bDX)_{\bar {\it x}}(-d)$, for each geometric point $\bar x $ on $\rmBG$. Moreover, in this case, $\rmE_2^{\it u, v} =0$ for $v <-2d$, or $v>0$ and also for $u<0$. It follows
that 
\be \begin{align}
\label{inf.cycle}
\rmH^0(\rmBG, R\pi_{ \rmX*}^{-2d}(\bDX)(-d) ) = \rmE_2^{0, -2d} &\cong \rmE_3^{0, -2d} \cong \cdots \rmE_{\infty}^{0, -2d} \cong \H^{-2d}(\rmE \rmG{\underset {\rmG} \times}  \rmX, \bDX(-d)) \\
=\rmH_{2d,d}(\rmE \rmG{\underset {\rmG} \times}  \rmX, \Q_{\ell}) &=\rmH_{2d,d}^{\rmG}(\rmX, \Q_{\ell}).\notag
\end{align} \ee
We next consider the special case where $\rmG$ is connected. Then one sees that the sheaf $\rmR\pi_{\rmX*}^{-2d}(\bDX)$ is constant on $\rmBG$, and 
therefore one obtains a class $[\rmX]_{\rmG} \in \rmH^0(\rmBG, R\pi_{ \rmX*}^{-2d}(\bDX)(-d))$. 
Now the identifications in ~\eqref{inf.cycle} shows that this is an infinite cycle in the above spectral sequence, 
 and therefore this defines
a fundamental class $[\rmX]_{\rmG} \in \rmH_{2d}(\rmE \rmG{\underset {\rmG} \times} \rmX, \Q_{\ell}(-d))= \H_{2d,d}^{\rmG}(\rmX, \Q_{\ell})$.
\vskip .1cm
When $\rmG$ is not connected, let $\rmG^o$ denote the connected component of the identity. Then the finite group $\rmG/\rmG^o$ acts on the stalks
of $\rmR\pi_{\rmX*}(\bDX)$. Since we are working with $\Q_{\ell}$-coefficients, one obtains isomorphisms 
\[\rmH^0(\rmBG, R\pi_{\rmX*}^{-2d}(\bDX)(-d)  )) \cong \rmH^0(\rmBG^o, R\pi_{\rmX*}^{-2d}(\bDX)(-d))^{\rmG/\rmG^o}, \mbox{ and }\]
\[\H^{\rm -2d}((\rmE \rmG{\underset {\rmG} \times}  \rmX, \bDX(-d)) \cong \H^{-2d}(\rmE \rmG^o{\underset {\rmG^o} \times}  \rmX, \bDX(-d))^{\rmG/\rmG^o}. \]
Therefore, the isomorphisms in ~\eqref{inf.cycle} extend to this case. It follows that one obtains a fundamental class $[\rmX]_{\rmG} \in \H^{-2d}(\rmE \rmG{\underset {\rmG} \times}  \rmX, \bDX(-d)) = \rmH_{2d,d}^{\rmG}(\rmX, \Q_{\ell})$.
\vskip .2cm
Next suppose that $i: \rmX \ra \tilde \rmX$ is a closed $\rmG$-equivariant immersion of a given $\rmG$-quasi-projective scheme into a smooth
$\rmG$-scheme $\tilde \rmX$. Let $\bDX$ (${\mathbb D}^{\rmG}_{\tilde \rmX}$) denote the dualizing complexes on $\rmE \rmG{\underset {\rmG} \times} \rmX$
($\rmE \rmG{\underset {\rmG} \times} \tilde \rmX$). (In view of the assumption that $\tilde \rmX$ is smooth, one may identify ${\mathbb D}^{\rmG}_{\tilde \rmX}$
 with $\Q_{\ell}[-2\tilde d](-\tilde d)$ where $\tilde \rmX$ is assumed to have pure dimension $\tilde d$.)
 Now the pairing ${\mathbb D}^{\rmG}_{\tilde \rmX} \otimes i_*{\rmR}i^!(\Q_{\ell}) \ra i_*\bDX$ induces a pairing of the corresponding hypercohomology spectral 
 sequences. Since the fundamental class $[\tilde \rmX]_{\rmG}$ is an infinite cycle, cap-product with this class then induces a map of
 spectral sequences from the spectral sequence in ~\eqref{equiv.coh.ss} for $K = i_*{\rmR}i^!(\Q_{\ell})$ to the corresponding spectral sequence for
  $K = \bDX$. Since this map of spectral sequences is an isomorphism at the $\rmE_2$-terms and both spectral sequences converge strongly, we obtain
  an isomorphism at the abutments, that is, an isomorphism
  \[ [\tilde \rmX]_{\rmG} \cap (\quad ):  \rmH^{s,t}_{\rmE \rmG{\underset {\rmG} \times}\rmX}(\rmE \rmG{\underset {\rmG} \times} \tilde \rmX, \Q_{\ell}) \ra \rmH_{2\tilde d-s, \tilde d-t}(\rmE \rmG{\underset {\rmG} \times}\rmX, \Q_{\ell}).\]
  \vskip .1cm
  Finally \cite[Theorem 1.6]{J20} provides a comparison of the 
   classifying spaces between the two models, one where $\rmBG$ is defined as an ind-scheme $\{\rmE\rmG^{gm,m}{\underset {\rmG} \times} (\Speck)|m\}$
   and the other is defined using the simplicial Borel construction. It shows that the two models have equivalent equivariant derived categories. Therefore,
   one may define equivariant Borel-Moore homology using either approaches and this verifies all the three required properties in Definition ~\ref{gen.equiv.coh.hom}.
  \vskip .1cm
  We skip the discussion for singular cohomology, as it follows along similar lines. (In fact one may also consult \cite[section 2]{J99}.)
  \end{proof}
 
\subsection{Riemann-Roch from forms of equivariant K-theory to equivariant Borel-Moore homology}
 \begin{theorem}
 \label{composite.RR.1}
 Let $\rm X \ra \{H_{s,t}(\rmX, \Gamma')|s,t\}$ denote a homology theory defined in the framework discussed in Definition ~\ref{gen.hom} so
 that it defines a generalized equivariant homology theory as in Definition ~\ref{gen.equiv.coh.hom}.
 \vskip .1cm
 (i) Let $f:\rmX \ra \rmY$ denote
a proper $\rmG$-equivariant map between two normal $\rmG$-quasi-projective schemes. Let $\tilde \rmG$ denote
either a ${\rm GL}_n$ or a finite product of groups of the form ${\rm GL}_n$ containing $\rmG$ as a closed subgroup-scheme. Then one obtains
the commutative diagram
\be \begin{equation}
   \label{RR.1.2}
\xymatrix{ {\pi_*{\bf G}(\rmX, \rmG)} \ar@<1ex>[r] \ar@<1ex>[d]^{Rf_*} &{\pi_*{\bf G}({\rm E}{\tilde \rmG}^{gm,m}{ \times}\rmX, \rmG)} \ar@<1ex>[r]^{\cong} \ar@<1ex>[d]^{Rf_*} & {\pi_*{\bf G}({\rm E}{\tilde \rmG}^{gm,m}{\underset {\rmG} \times}\rmX)} \ar@<1ex>[d]^{Rf_*} \ar@<1ex>[r]^{\tau_{\rmX,m}} &{ \rmH_{*, \bullet}({\rm E}{\tilde \rmG}^{gm,m}{\underset {\rmG} \times}\rmX, \Gamma')_{{\mathbb Q}}} \ar@<1ex>[d]^{f_*}\\
              {\pi_*{\bf G}(\rmY, \rmG)} \ar@<1ex>[r] & {\pi_*{\bf G}({\rm E}{\tilde \rmG}^{gm,m}{ \times}\rmY, \rmG)} \ar@<1ex>[r]^{\cong}&  {\pi_*{\bf G}({\rm E}{\tilde \rmG}^{gm,m}{\underset {\rmG} \times}\rmY)}  \ar@<1ex>[r]^{\tau_{\rmY,m}} & { \rmH_{*, \bullet}({\rm E}{\tilde \rmG}^{gm,m}{\underset {\rmG} \times}\rmY, \Gamma')_{{\mathbb Q}} }}
\end{equation} \ee
for any fixed integer $m \ge 0$. 
\vskip .1cm
(ii) Assume in addition to the hypotheses in (i) that the map $f$ is also perfect, for example, a map that is a local complete intersection morphism.
Then one obtains the commutative diagram
\be \begin{equation}
 \label{RR.2.2}
\xymatrix{ {\pi_*{\bKH}(\rmX, \rmG)} \ar@<1ex>[r] \ar@<1ex>[d]^{Rf_*} &{\pi_*{\bKH}({\rm E}{\tilde \rmG}^{gm,m}{ \times}\rmX, \rmG)} \ar@<1ex>[r]^{\cong} \ar@<1ex>[d]^{Rf_*}  & {\pi_*{\bKH}({\rm E}{\tilde \rmG}^{gm,m}{\underset {\rmG} \times}\rmX)} \ar@<1ex>[d]^{Rf_*} \ar@<1ex>[r]^{\tau_{\rmX,m}} &{ \rmH_{*, \bullet}({\rm E}{\tilde \rmG}^{gm,m}{\underset {\rmG} \times}\rmX, \Gamma')_{{\mathbb Q}}} \ar@<1ex>[d]^{f_*}\\
              {\pi_*{\bKH}(\rmY, \rmG)} \ar@<1ex>[r] &{\pi_*{\bKH}({\rm E}{\tilde \rmG}^{gm,m}{ \times}\rmY, \rmG)} \ar@<1ex>[r]^{\cong} &  {\pi_*{\bKH}({\rm E}{\tilde \rmG}^{gm,m}{\underset {\rmG} \times}\rmY)}  \ar@<1ex>[r]^{\tau_{\rmY,m}} & { \rmH_{*, \bullet}({\rm E}{\tilde \rmG}^{gm,m}{\underset {\rmG} \times}\rmY, \Gamma')_{{\mathbb Q}} }}
\end{equation} \ee
for any fixed integer $m \ge 0$. 
\end{theorem}
 \begin{proof}
  The commutativity of the first squares in both diagrams ~\ref{RR.1.2} and ~\ref{RR.2.2} is clear from Theorems ~\ref{RR.Gth} and ~\ref{RR.KH}. The commutativity of the rest of the 
  diagram ~\ref{RR.1.2} is also clear since the induced map ${\rm E}{\tilde \rmG}^{gm,m}{\underset {\rmG} \times}\rmX \ra {\rm E}{\tilde \rmG}^{gm,m}{\underset {\rmG} \times}\rmY$ is
 a proper map. The commutativity of the second square in (ii) is also clear.
 To prove the commutativity of the last square in (ii), one first observes that sending a perfect complex to itself, but viewed
 as a pseudo-coherent complex (or a complex of quasi-coherent $\O$-modules with bounded coherent cohomology) defines a natural map
 $\bKH(\quad) \ra {\bf G}(\quad)$ of presheaves of spectra. Under the assumption that the map $f$ is perfect and $\rmG$-equivariant,
 one may see that this map commutes with push-forward by $f_*$, resulting in the homotopy commutative square:
\[ \xymatrix{ {\bKH({\rm E}{\tilde \rmG}^{\it gm,m}{\underset {\rmG} \times}\rmX)} \ar@<1ex>[r] \ar@<1ex>[d] ^{Rf_*} & {{\bf G}({\rm E}{\tilde \rmG}^{gm,m}{\underset {\rmG} \times}\rmX)} \ar@<1ex>[d]^{Rf_*}\\
             {\bKH({\rm E}{\tilde \rmG}^{\it gm,m}{\underset {\rmG} \times}\rmY)} \ar@<1ex>[r] & {{\bf G}({\rm E}{\tilde \rmG}^{gm,m}{\underset {\rmG} \times}\rmY)} }
\]
 One may take the homotopy groups of the above diagram and compose the horizontal maps in the resulting square with the corresponding horizontal maps in the third square of the diagram ~\eqref{RR.1.2} to obtain the commutative
 square forming the third square in the diagram ~\eqref{RR.2.2}. 
 \end{proof}
 \vskip .1cm
 \begin{lemma}
 \label{RR.extends.inverse.systems.1}
 Assume the hypotheses as in Definitions ~\ref{gen.coh}, ~\ref{gen.hom} and ~\ref{gen.equiv.coh.hom}  on the homology/cohomology theories and that $\rmX$ is a $\rmG$-quasi-projective normal scheme, for a linear algebraic group. Then the Riemann-Roch transformation
 $\tau_{\rmX, m}: \pi_*{\bf G}({\rm E}{\tilde \rmG}^{gm,m}{\underset {\rmG} \times}\rmX) \ra \rmH_{*, \bullet}({\rm E}{\tilde \rmG}^{gm,m}{\underset {\rmG} \times}\rmX, \Gamma')_{{\mathbb Q}}$
 extends to a Riemann-Roch transformation
 $\tau_{\rmX} = \limm \tau_{\rmX, m}: \limm \pi_*{\bf G}({\rm E}{\tilde \rmG}^{gm,m}{\underset {\rmG} \times}\rmX) \ra \limm \rmH_{*, \bullet}({\rm E}{\tilde \rmG}^{gm,m}{\underset {\rmG} \times}\rmX, \Gamma')_{{\mathbb Q}}$.
\end{lemma}
\begin{proof}
 One first recalls from \cite{Gi}, how the Riemann-Roch transformation is defined. Making use of the assumption that $\rmX$ is $\rmG$-quasi-projective, one
 may find a smooth scheme $\tilde \rmX$ provided with a $\rmG$-action and containing $\rmX$ as a $\rmG$-stable closed subscheme. 
 Therefore, one first defines the (local equivariant) Chern-character
 \[ch_m: \pi_*{\bf G}({\rm E}{\tilde \rmG}^{gm,m}{\underset {\rmG} \times}\rmX) \cong \pi_*{\bK}_{{\rm E}{\tilde \rmG}^{gm,m}{\underset {\rmG} \times}\rmX}({\rm E}{\tilde \rmG}^{gm,m}{\underset {\rmG} \times}\tilde \rmX) {\overset {ch_m} \ra} \rmH^{*, \bullet}_{{\rm E}{\tilde \rmG}^{gm,m}{\underset {\rmG} \times}\rmX}({\rm E}{\tilde \rmG}^{gm,m}{\underset {\rmG} \times}\tilde \rmX, \Gamma)_{\Q}. \]
 These are clearly compatible as $m$-varies. Therefore, we obtain an induced local Chern-character:
 \[ch =\limm ch_m: \limm \pi_*{\bf G}({\rm E}{\tilde \rmG}^{gm,m}{\underset {\rmG} \times}\rmX) \cong \limm \pi_*{\bK}_{{\rm E}{\tilde \rmG}^{gm,m}{\underset {\rmG} \times}\rmX}({\rm E}{\tilde \rmG}^{gm,m}{\underset {\rmG} \times}\tilde \rmX)\]
 \[{\overset {\limm ch_m} \ra} \limm \rmH^{*, \bullet}_{{\rm E}{\tilde \rmG}^{gm,m}{\underset {\rmG} \times}\rmX}({\rm E}{\tilde \rmG}^{gm,m}{\underset {\rmG} \times}\tilde \rmX, \Gamma)_{\Q}. \]
 In view of the assumption that for each pair of  fixed
 integers $s, t$, the inverse system
 \[\{\rmH^{s,t}(\rmE \tilde \rmG^{gm,m}\times _{\rmG}\tilde \rmX, \Gamma)_{\Q} |m\}\]
 stabilizes as $m \ra \infty$, one may
 readily define the  Todd class of the tangent bundle to the scheme $\tilde \rmX$ with values in the inverse limit $\limm \Pi_{s,t}\rmH^{s,t}(\rmE \rmG^{gm,m}\times _{\rmG}\tilde \rmX, \Gamma)_{\Q}$.
 In view of the presence of an equivariant fundamental class $[\tilde \rmX]_{\rmG}$ as in Definition ~\eqref{gen.equiv.coh.hom}, 
 and since $\{ \rmH_{i+2\tilde d_m -2g,j+\tilde d_m-g}(\rmE \tilde \rmG^{gm,m}\times _{\rmG}\rmX, \Gamma') |m\}$  stabilizes as $m \ra \infty$ (see
 Definition ~\ref{gen.equiv.coh.hom}(ii)), one may now
 define the Riemann-Roch transformation, by cupping the local-Chern-character homomorphism with the above equivariant Todd class (so that the composite map takes values 
 in $\limm \Pi_{s,t}\rmH^{s,t}_{\rmE \tilde \rmG^{gm,m}\times _{\rmG}\rmX}(\rmE \tilde \rmG^{gm,m}\times _{\rmG}\tilde \rmX, \Gamma)_{\Q}$), followed by capping with
 the equivariant fundamental class $[\tilde \rmX]_{\rmG}$. It is now straight-forward to verify the resulting Riemann-Roch transformation is the inverse limit
 \[\limm  \tau_{\rmX, m}: \limm \pi_*{\bf G}({\rm E}{\tilde \rmG}^{gm,m}{\underset {\rmG} \times}\rmX) \ra \limm \rmH_{*, \bullet}({\rm E}{\tilde \rmG}^{gm,m}{\underset {\rmG} \times}\rmX, \Gamma')_{{\mathbb Q}}.\]
 \end{proof}
 
\vskip .1cm

\begin{corollary}
 \label{RR.extends.inverse.systems.2}
 Assume in addition to the hypotheses in Lemma ~\ref{RR.extends.inverse.systems.1} that $\rmY$ is a normal $\rmG$-quasi-projective scheme and that 
 $f: \rmX \ra \rmY$ is a $\rmG$-equivariant proper map, with $\rmX$ also a normal $\rmG$-quasi-projective scheme. Then the following diagram commutes for every choice of $m$ sufficiently large, and where the slant maps in the right square are defined
 using the Gysin maps as in ~\eqref{Gysin.eq.homology} and where $c= dim(\rmE \tilde \rmG^{gm,m+1}) - dim(\rmE \tilde \rmG^{gm,m})$:
 \be \begin{equation}
 \label{RR.inverse.systems}
 \xymatrix{  {\pi_*\bG(\rmX, \rmG)}\ar@<1ex>[ddd]^{Rf_*}\ar@<1ex>[dr]^{id} \ar@<1ex>[rrr]^{\tau_{\rmX}^{m+1}}  &&&  {\Pi_{\it u, v}\rmH_{\it u+2c, v+c}(\rmE \tilde \rmG^{gm,m+1}\times _{\rmG} \rmX, \Gamma')_{\Q}} \ar@<1ex>[dl]^{} \ar@<1ex>[ddd]^{f_*}\\
  &{\pi_*\bG(\rmX, \rmG)} \ar@<1ex>[r]^(.4){\tau_{\rmX}^m} \ar@<1ex>[d]^{Rf_*} & {\Pi_{\it u, v}\rmH_{\it u, v}(\rmE \tilde \rmG^{gm,m}\times _{\rmG} \rmX, \Gamma')_{\Q}}  \ar@<1ex>[d]^{f_*} \\
  &{\pi_*\bG(\rmY, \rmG)} \ar@<1ex>[r]^(.4){\tau_{\rmY}^m} &{\Pi_{\it u, v}\rmH_{\it u, v}(\rmE \tilde \rmG^{gm,m}\times _{\rmG} \rmY, \Gamma')_{\Q}}\\
  {\pi_*\bG(\rmY, \rmG)} \ar@<1ex>[ur]^{id}  \ar@<1ex>[rrr]^{\tau_{\rmY}^{m+1}}   &&&    {\Pi_{\it u, v}\rmH_{\it u+2c, v+c}(\rmE \tilde \rmG^{gm,m+1}\times _{\rmG} \rmY, \Gamma')_{\Q}} \ar@<1ex>[ul]_{}.}
 \end{equation} \ee
 \end{corollary}
\begin{proof} The inner and outer squares clearly commute by Theorem ~\ref{composite.RR.1}. Lemma ~\ref{RR.extends.inverse.systems.1}
shows that the Riemann-Roch transformations extend to define a Riemann-Roch transformation with  the inverse limits considered
there as the targets. This shows the squares in the top and bottom commute. Therefore, it suffices to show the right square commutes.
In view of our assumption, one may factor the map $f$ as the composition of a $\rmG$-equivariant closed immersion
 $i:\rm X \ra \rmY \times {\mathbb P}^n$ for some $n>>0$ and the projection $\pi: \rmY \times {\mathbb P}^n \ra \rmY$. Then 
 the well-known formula for the homology of the projective spaces enables one to readily 
  check that the corresponding diagram commutes when $f$ is replaced by $\pi$.
  \vskip .1cm
 Therefore, it remains to consider the case of a closed immersion $i: \rmX \ra \rmY$. Then, in view of our assumptions on
 $\rmX$ and $\rmY$ as $\rmG$-quasi-projectives schemes, we may find closed $\rmG$ equivariant immersions $\rmX \ra \tilde \rmX$
 and $\rmY \ra \tilde \rmY$, together with a closed $\rmG$-equivariant immersion $\tilde \rmX \ra \tilde \rmY$ so that the square
 \[ \xymatrix{{\rmX} \ar@<1ex>[r] \ar@<1ex>[d] & {\rmY} \ar@<1ex>[d]\\
              {\tilde \rmX} \ar@<1ex>[r] & \tilde \rmY}
 \]
 commutes. Then the commutativity of the
 right square in the diagram ~\eqref{RR.inverse.systems} amounts to the commutativity of the diagram (where $c'= dim(\tilde \rmY) - dim(\tilde \rmX)$):
 \vskip .1cm
 \[\xymatrix{{\rmH^{*, \bullet}_{\rmE \tilde \rmG^{gm,m}\times _{\rmG} \rmX}(\rmE \tilde \rmG^{gm,m}\times _{\rmG} \tilde \rmX, \Gamma)} \ar@<1ex>[d]^{\alpha_m} & {\rmH^{*, \bullet}_{\rmE \tilde \rmG^{gm,m+1}\times _{\rmG} \rmX}(\rmE \tilde \rmG^{gm,m+1}\times _{\rmG} \tilde \rmX, \Gamma)} \ar@<1ex>[l]^{i_{\rmX}^*} \ar@<1ex>[d]^{\alpha_{m+1}}\\
 {\rmH^{*+2c', \bullet+c'}_{\rmE \tilde \rmG^{gm,m}\times _{\rmG} \rmX}(\rmE \tilde \rmG^{gm,m}\times _{\rmG} \tilde \rmY, \Gamma)} \ar@<1ex>[d]^{\beta_m} & {\rmH^{*+2c', \bullet+c'}_{\rmE \tilde \rmG^{gm,m+1}\times _{\rmG} \rmX}(\rmE \tilde \rmG^{gm,m+1}\times _{\rmG} \tilde \rmY, \Gamma)} \ar@<1ex>[l]^{i_{\rmY}^*} \ar@<1ex>[d]^{\beta_{m+1}}\\
 {\rmH^{*+2c', \bullet+c'}_{\rmE \tilde \rmG^{gm,m}\times _{\rmG} \rmY}(\rmE \tilde \rmG^{gm,m}\times _{\rmG} \tilde \rmY, \Gamma)} & {\rmH^{*+2c', \bullet+c'}_{\rmE \tilde \rmG^{gm,m+1}\times _{\rmG} \rmY}(\rmE \tilde \rmG^{gm,m+1}\times _{\rmG} \tilde \rmY, \Gamma)} \ar@<1ex>[l]^{i_{\rmY}^*}.}\]
 Let the normal bundle associated to the closed immersion
 $\rmE \tilde \rmG^{gm,m}\times _{\rmG} \tilde \rmX \ra \rmE \tilde \rmG^{gm,m}\times _{\rmG} \tilde  \rmY$ be denoted $\cN_m$,
 while we let the normal bundle associated to the closed immersion
 $\rmE \tilde \rmG^{gm,m+1}\times _{\rmG} \tilde \rmX \ra \rmE \tilde \rmG^{gm,m+1}\times _{\rmG} \tilde  \rmY$ be denoted $\cN_{m+1}$.
The map denoted $\alpha_m$ ($\alpha_{m+1}$) is cup-product with the Koszul-Thom class $\lambda_{-1}(\cN_m)$ ($\lambda_{-1}(\cN_{m+1})$, \res).
 Here $i_{\rmX}$ denotes the closed immersion $\rmE \tilde \rmG^{gm,m}\times _{\rmG} \tilde \rmX \ra \rmE \tilde \rmG^{gm,m+1}\times _{\rmG} \tilde \rmX$
 while $i_{\rmY}$ denotes the closed immersion $\rmE \tilde \rmG^{gm,m}\times _{\rmG} \tilde \rmY \ra \rmE \tilde \rmG^{gm,m+1}\times _{\rmG} \tilde \rmY$, \res.
 The commutativity of the top square follows from the observation that $i_Y^*(\lambda_{-1}(\cN_{m+1})) = \lambda_{-1}(\cN_m)$.
 The maps denoted $\beta_m$ and $\beta_{m+1}$ are the obvious maps, and therefore the commutativity of the bottom square is clear.
\end{proof}
\vskip .2cm \noindent
{\bf Proof of Theorem ~\ref{composite.RR.}} Recall that for a tower of fibrations $\{\cX_m \leftarrow \cX_{m+1}|m\}$, the map 
 $\pi_n(\holimm \cX_m) \ra \limm \pi_n(\cX_n)$ is natural for such towers, and therefore compatible with maps of coherently homotopy commutative diagrams of such towers $\{\cX_m \leftarrow \cX_{m+1}|m\}$. We apply this to the 
 commutative diagrams in ~\eqref{RR1.2} and ~\eqref{RR2.2} to obtain the commutative diagrams:
 \fontsize{8}{12}
 \begin{equation}
 \label{RR.3.1}
 \xymatrix{ {\pi_*{\bf G}(\rmX, \rmG)} \ar@<1ex>[r] \ar@<1ex>[d]^{Rf_*} & {\pi_*(\holimm {\bf G}(\rmX, \rmG)\compl_{\rho_{{\tilde \rmG}, \alpha(m)}})} \ar@<1ex>[r]^{\cong} \ar@<1ex>[d]^{Rf_*} &  {\pi_*(\holimm {\bf G}({\rm E}{\tilde \rmG}^{gm,m}{\underset {\rmG} \times}\rmX))} \ar@<1ex>[d]^{Rf_*}\ar@<1ex>[r] & {\limm \pi_*({\bf G}( {\rm E} {\tilde \rmG}^{gm, m} \times \rmX))} \ar@<1ex>[d]^{\limm Rf_*}\\
              {\pi_*{\bf G}(\rmY, \rmG)} \ar@<1ex>[r]  & {\pi_*(\holimm {\bf G}(\rmY, \rmG)\compl_{\rho_{{ \rmG}, \alpha(m)}})} \ar@<1ex>[r]^{\cong}  &  {\pi_*(\holimm {\bf G}({\rm E}{\tilde \rmG}^{gm,m}\times \rmY, \rmG))} \ar@<1ex>[r]^{} & {\limm \pi_*({\bf G}({\rm E}{\tilde \rmG}^{gm,m}{\underset {\rmG} \times}\rmY)}, \mbox{ and}}
 \end{equation}
 \begin{equation}
  \label{RR.3.2}
\xymatrix{ {\pi_*{\bKH}(\rmX, \rmG)} \ar@<1ex>[r] \ar@<1ex>[d]^{Rf_*} & {\pi_*(\holimm {\bKH}(\rmX, \rmG)\compl_{\rho_{{\tilde \rmG}, \alpha(m)}})} \ar@<1ex>[r]^{\cong} \ar@<1ex>[d]^{Rf_*} &  {\pi_*(\holimm {\bKH}({\rm E}{\tilde \rmG}^{gm,m}{\underset {\rmG} \times}\rmX))} \ar@<1ex>[d]^{Rf_*} \ar@<1ex>[r]& {\limm \pi_*({\bKH}( {\rm E} {\tilde \rmG}^{gm, m} \times \rmX))} \ar@<1ex>[d]^{\limm Rf_*}\\
              {\pi_*{\bKH}(\rmY, \rmG)} \ar@<1ex>[r]  & {\pi_*(\holimm {\bKH}(\rmY, \rmG)\compl_{\rho_{{ \rmG}, \alpha(m)}})} \ar@<1ex>[r]^{\cong}  &  {\pi_*(\holimm {\bKH}({\rm E}{\tilde \rmG}^{gm,m}\times \rmY, \rmG))} \ar@<1ex>[r]^{} & {\limm \pi_*({\bKH}({\rm E}{\tilde \rmG}^{gm,m}{\underset {\rmG} \times}\rmY)}.}
\end{equation}             
\normalsize          
Lemma ~\ref{RR.extends.inverse.systems.1} and Corollary ~\eqref{RR.extends.inverse.systems.2} now show that the commutativity of the 
 diagram ~\ref{RR.1.1} follows from the commutativity of the diagram ~\ref{RR.3.1}. These observations complete the proof of the first statement.
 Then one deduces the commutativity of the diagram ~\ref{RR.2.1} from that of the diagram ~\ref{RR.3.2}, just as in
 statement (i). \qed
 \vskip .2cm 
 {\it Observe that Corollary ~\ref{finite.group.case} follows readily} by imbedding the finite group as closed subgroup of a ${\rm GL}_n$ as discussed in the lines before the statement of this Corollary in the introduction.
\vskip .1cm
\begin{remark} It may be important to point out that, in view of the difficulties with 
usual Atiyah-Segal type completion, (for example as in \cite{K}), the only equivariant Riemann-Roch theorems known {\it until now} are only for
Grothendieck groups, unless one assumes that both $\rmX$ and $\rmY$ are projective smooth schemes. 
The above Riemann-Roch theorems extend the Riemann-Roch theorems of Gillet (as in \cite{Gi}) for higher algebraic K-theory and G-theory to the equivariant framework and
to large classes of schemes with group action by linear algebraic groups.
\end{remark}

\section{Lefschtez-Riemann-Roch and applications}
We conclude this discussion by providing a proof of the {\it Lefschetz-Riemann-Roch and coherent trace formula} discussed
 in Theorem ~\ref{LRR}, which extends similar formulae
of Thomason for equivariant topological G-theory: see \cite[Theorem 6.4]{Th86-1}. But first we prove the following localization theorem for actions by diagonalizable groups.
\vskip .1cm
Let $\rmT$ denote a split torus and let $\rmM$ denote its group of characters, with ${\mathbb Z}[\rmM]$ its associated
group ring. For each prime ${\rm p}$ of $\bZ[\rmM]$, there is a minimal subgroup $\rmM_p$ of $\rmM$ so that the prime $\rmp$ is the inverse
 image of a prime in $\bZ[\rmM/\rmM_p]$. Let $\rmT_p$ denote the diagonalizable subgroup ${\rm D}(\rmM/\rmM_p)$. Under this correspondence
 the $(0)$-ideal corresponds to the torus $\rmT$ itself.
\begin{theorem}
 \label{localization.torus.act}
 Let $\rmX$ denote a scheme of finite type over $\k$ provided with an action by $\rmT$. Let $i:\rmX^{\rmT}\ra \rmX$ denote
the closed immersion from the corresponding fixed-point scheme into $\rmX$. Then 
\[(\oplus_n \pi_n(\holimm{\bf G}(\rmT^{gm,m}\times_{\rmT}(\rmX - \rmX^{\rmT}), \rmT)))_{(0)} \cong \{0\} \mbox{ and } (\oplus_n \limm \pi_n({\bf G}(\rmT^{gm,m}\times_{\rmT}(\rmX - \rmX^{\rmT}), \rmT)))_{(0)} \cong \{0\}.\]
\end{theorem}
\begin{proof}
 We will assume that conclusion holds for all
schemes $\rmZ$ on which $\rmT$ acts with no fixed points and where the Krull dimension of Z is less than the Krull dimension of
$\rmX - \rmX^{\rmT}$. Next assume that $\rmX - \rmX^{\rmT}$ is connected. Let $\rmU$ denote a non-empty open and $\rmT$-stable
 subscheme of $\rmX - \rmX^{\rmT}$. Now consider the localization sequence, where $\rmZ = (\rmX - \rmX^{\rmT}) - \rmU$:
\be \begin{multline}
     \begin{split}
    \label{local.seq}
  \cdots \ra (\oplus_n \pi_n(\holimm {\bf G}(\rmT^{gm,m}\times_{\rmT}\rmZ, \rmT)))_{(0)} \ra (\oplus_n \pi_n(\holimm {\bf G}(\rmT^{gm,m}\times_{\rmT}(\rmX - \rmX^{\rmT}), \rmT)))_{(0)}\\
  \ra (\oplus_n \pi_n(\holimm {\bf G}(\rmT^{gm,m}\times_{\rmT}\rmU, \rmT)))_{(0)} \ra \cdots.
\end{split}  \end{multline} \ee
This localization sequence and the inductive assumption shows that it suffices to prove the theorem 
with $\rmX - \rmX^{\rmT}$ replaced by any open and $\rmT$-stable subscheme $\rmU$ of $\rmX- \rmX^{\rmT}$. At this point one
invokes the {\it Torus generic slice theorem} in \cite[Proposition 4.10]{Th86-2} to conclude that the following hold:
\begin{enumerate}[\rm(i)]
\item there exists a nonempty
open and $\rmT$-stable subscheme $\rmU$ of $\rmX$ and a sub-torus $\rmT'$ of $\rmT$, with quotient torus 
$\rmT/\rmT'= \rmT''$ which is non-trivial, and  so that
$\rmT$ acts on $\rmU$ through its quotient $\rmT''$, 
\item the resulting action of $\rmT''$ on $\rmU$ is free and 
\item $\rmU$ is
isomorphic as a $\rmT$-scheme to $\rmT'' \times \rmU/\rmT$, with $\rmU/\rmT$ denoting a geometric quotient.
\end{enumerate}
Therefore, $\rmE\rmT^{gm,m}\times_{\rmT} \rmU \cong \rmE\rmT'' \times \rmB\rmT'^{gm, m}  \times \rmU/\rmT$, for all $m$.
It follows that $\holimm {\bf G}(\rmE\rmT^{gm,m}\times_{\rmT}\rmU, \rmT) \simeq \holimm {\bf G}(\rmB\rmT'^{gm,m} \times \rmU/\rmT)$ and hence $(\oplus_n \pi_n(\holimm {\bf G}(\rmE\rmT^{gm,m}\times_{\rmT}\rmU, \rmT)))_{(0)} \cong \{0\}$. 
Therefore, by the inductive assumption on the Krull dimension of $\rmX$ and by the localization sequence ~\eqref{local.seq},
\[(\oplus_n \pi_n(\holimm {\bf G}(\rmE\rmT^{gm,m}\times_{\rmT}(\rmX - \rmX ^{\rmT}), \rmT)))_{(0)} \cong \{0\}.\]
In case $\rmX- \rmX^{\rmT}$ is not connected, an ascending induction on the number of connected components
and a similar argument using the same localization sequence will prove the same result in general.
One may observe that $\oplus_n \limm \pi_n({\bf G}(\rmE\rmT^{gm,m}\times_{\rmT}(\rmX - \rmX^{\rmT}), \rmT))$ is a module
over $\oplus_n \pi_n(\holimm{\bf G}(\rmE\rmT^{gm,m}\times_{\rmT}(\rmX - \rmX^{\rmT}), \rmT)))$ to obtain the 
second equality. 
\end{proof}
\begin{corollary}
 Let $\rmZ$ denote a regular scheme of finite type over $\k$ provided with the action of a split torus $\rmT$. Let $i: \rmZ^{\rmT} \ra \rmZ$ denote
 the closed immersion of the corresponding fixed point scheme and let $\cN$ denote the conormal sheaf associated to the closed immersion $i$.
 Then the class $\lambda_{-1}(\cN)$ is a unit in $\pi_{0}(\bK(\rmZ^{\rmT}, \rmT))_{(0)}$ and the localized Gysin map
 \[i_*: \pi_*({\bf G}(\rmZ^{\rmT}, \rmT))_{(0)} \ra \pi_*({\bf G}(\rmZ, \rmT))_{(0)}\]
 is injective  with a left-inverse given by the map ${\mathcal F} \mapsto {\rm L} i^*(\mathcal F ) \cap \lambda_{-1}(\cN)^{-1}$.
 Moreover the localized Gysin map
 \[i_*:\pi_*(\holimm {\bf G}(\rmE \rmT^{gm,m} \times_{\rmT}\rmZ^{\rmT}, \rmT))_{(0)}  \ra \pi_*(\holimm {\bf G}(\rmE \rmT^{gm,m} \times_{\rmT}\rmZ, \rmT))_{(0)}  \]
 is an isomorphism, with a left-inverse given by the map ${\mathcal F} \mapsto {\rm L} i^*(\mathcal F ) \cap \lambda_{-1}(\cN)^{-1}$.
\end{corollary}
\begin{proof} It is well-known that the composition of the Gysin map $i_*$ and ${\rm L}i^*$ is multiplication by the class $\lambda_{-1}(\cN)$.  See \cite[Lemma 6.3]{Th86-1} for a proof that $\lambda_{-1}(\cN)$ is a unit 
in $\pi_{0}(\bK(\rmZ^{\rmT}, \rmT))_{(0)}$. The first
statement follows from these observations.
That the Gysin map in the second statement is an isomorphism follows from
Theorem ~\ref{localization.torus.act}, using the localization sequence for the open immersion 
$\rmZ - \rmZ^{\rmT} \ra \rmZ$, with complement $\rmZ^{\rmT}$.
 \end{proof}
\vskip .1cm \noindent
{\bf Proof of Theorem ~\ref{LRR}: The Lefschetz-Riemann-Roch and the Coherent trace formula.}
We quickly recall the framework of this theorem.
  Let $f: \rmX \ra \rmY$ denote a $\rmT$-equivariant proper map between two schemes of finite type over $\k$ and provided with
  actions by split torus $\rmT$. Assume that $f$ factors $\rmT$-equivariantly as $\pi \circ i_1$, where $i_1: \rmX \ra \rmZ$ is
  closed immersion into a regular $\rmT$-scheme $\rmZ$ followed by a proper $\rmT$-equivariant map $\pi: \rmZ \ra \rm Y$. Let $i: \rmZ^{\rmT} \ra \rmZ$ denote
  the closed immersion of the fixed point scheme of $\rmZ$ into $\rmZ$.
  Let ${\mathcal F}$ denote a $\rmT$-equivariant coherent sheaf on $\rmX$ and let 
  \[{\mathcal G} = Li^*(i_{1*}({\mathcal F})){\underset {\O_{\rmZ^{\rmT}}} \otimes} (\lambda_{-1}(\cN))^{-1}.\]
 One may readily reduce to the case $\rmZ = \rmX$, by replacing ${\mathcal F}$ by $i_{1*}({\mathcal F})$. Therefore, we will assume $f = \pi$ in what follows. Then the
 proof of the equality in (i) results from the commutative diagram:
 \be \begin{equation}
  \label{Coh.tr.diagm}
 \xymatrix{  {\pi_*\bG(\rmZ, \rmT)_{(0)}} \ar@<1ex>[dd]^{R\pi_*} \ar@<1ex>[dr] \ar@<1ex>[rrr]^{i^* (\quad) \cap \lambda_{-1}(\cN)^{-1}}  &&&  {\pi_*({\bf G}(\rmZ^{\rmT}, \rmT))_{(0)}} \ar@<1ex>[dl] \ar@<1ex>[dd]^{\pi^{\rmT}_*}\\
  &{\pi_*( \bG(\rmE\rmT\times_{\rmT}\rmZ))_{(0)}} \ar@<1ex>[r]^{i^*(\quad) \cap \lambda_{-1}(\cN)^{-1}} \ar@<1ex>[dd]^{R\pi_*} & {\pi_*(\bG(\rmE\rmT\times_{\rmT}\rmZ^{\rmT}))_{(0)}}  \ar@<1ex>[l]^{i_*}_{\cong} \ar@<1ex>[dd]^{\pi^{\rmT}_*} \\
 {\pi_*\bG(\rmY, \rmT)_{(0)}} \ar@<1ex>[dr]    &&&   {\pi_*(\bG(\rmY^{\rmT}, \rmT))_{(0)}}  \ar@<1ex>[dl]\\ 
   &{\pi_*(\bG(\rmE\rmT\times_{\rmT}\rmY)_{(0)}}  &{\pi_*(\bG(\rmE\rmT\times_{\rmT}\rmY^{\rmT})_{(0)}} \ar@<1ex>[l]^{i_*} }
 \end{equation} \ee
 where we have let $\bG(\rmE\rmT\times_{\rmT}V) = \holimm \bG(\rmE\rmT^{gm.m}\times_{\rmT}V)$ for a scheme
 $\rmV$ with an action by $\rmT$. 
\vskip .1cm
 If $\rmX$ is $\rmT$-equivariantly quasi-projective, one may factor the map $f$ as the composite of the closed immersion $\rmX \ra \rmY \times {\mathbb P}^n$
 for some large enough projective space on which the torus $\rmT$ acts and the projection $\rmY \times {\mathbb P}^n \ra \rmY$. Therefore, when $\rmY$ is regular, one may take $\rmZ = \rmY \times {\mathbb P}^n$ in this case.
 When $\rmX$ is normal, it is $\rmT$-quasi-projective by \cite[Theorem  2.5]{SumII}. This proves (ii).
 \vskip .1cm
 Next we consider (iii). A key observation now is that, since the $\rmT$-action on $\rmY$ is assumed to be trivial,  $\pi_0(\bG(\rmE\rmT\times_{\rmT}\rmY)) = {\mathbb Z}[\rmT]\compl_{\rmI_{\rmT}} \otimes \pi_0(\bG(\rmY))$. 
 This follows from the observation that $\rmB\rmT^{gm,m} = ({\mathbb P}^m)^n$, if $\rmT ={\mathbb G}_m^n$, and therefore $\limm^1\pi_1(\bG(\rmB\rmT^{gm,m} \times \rmY))=0$.
 Moreover
 $ {\mathbb Z}[\rmT]\compl_{\rmI_{\rmT}} $ is a Laurent formal power series ring, and is therefore torsion-free over ${\mathbb Z}[\rmT]$. 
 As the classes that are inverted on taking the localization at the prime ideal $(0)$ are non-zero divisors, equality in ~\eqref{equality}
  holds in $\pi_0(\holimm {\bf G}(\rmE\rmT^{gm,m}\times_{\rmT} \rmY))$, if it holds after localization at the ideal $(0)$.
 (See \cite[Corollary 6.5]{Th86-2}, for example.) This  completes the proof of (iii). (iv) and the extension to
  equivariant homotopy K-theory in (v) are now clear.
\qed
\vskip .2cm \noindent
{\bf Proof of Corollary} ~\ref{LRR.cor} The first two statement should be clear by applying Riemann-Roch transformations into an equivariant Borel-Moore homology theory. 
The third statement follows simply by taking ${\mathcal F}$ in Corollary ~\ref{LRR.cor}(i) to be $\O_{\rmX}$.
\vskip .2cm  \noindent
{\bf Proof of Corollary} ~\ref{Dem.toric} (i) This follows readily from Theorem ~\ref{LRR}(iv) by taking $\rmX$ there to be the Schubert variety $\rmX_{\it w}$ and $\rmY = \Speck$.
Observe that the $\rmT$-fixed points on $\rmX_{\it w}$ correspond to $\{{\it w}' \in {\rm W}|{\it w'} \le {\it w}\}$. Therefore, expanding out ${\mathcal G}$, one gets the 
term $\rmH^0(\rmX_{\it w}, {\mathcal G}) = \Sigma_{\it w' \le w} a({\it w', w})e^{{\it w}'\lambda}$, for some coefficients ${\rm a}({\it w', w})$. To show that this gives the required expression, one may invoke
\cite[Lemma 6.13]{Th86-1}. Though the above argument is similar to the one in \cite[pp. 541-542]{Th86-1},  
it needs to be pointed out that as Thomason works there with G-theory with finite coefficients and with 
the Bott-element inverted, he has to work considerably harder to show the required identity holds in the integral representation ring $\rmR(\rmT)$.
Our proof is more direct, as we are working with the $\pi_0(\bG(\Speck, \rmT))$. Observe that when ${\it w}$ is the longest element in the Weyl group ${\rm W}$,
$\rmX_{\it w} = \rmG/\rmB$: in this case the above formula is just the Weyl character formula. \qed
\vskip .1cm

  \section{Computing the homotopy groups of the derived completion in terms of equivariant Borel-Moore motivic cohomology}
\vskip .2cm \noindent
{\bf Proof of Theorem ~\ref{compute.der.compl}}
  Assume that $\rmX$ is a $\rmG$-scheme as in Definition ~\ref{G.quasiproj}, with $\rmG= \rmT$ a split torus. Then the approximation $\rmE\rmG^{gm,m}{\underset {\rmG} \times}\rmX$ is always a scheme,  and it is smooth when $\rmX$ is smooth.
  We will first assume that $\rmX$ is in fact smooth.
  Therefore, for a fixed value of $-s-t$, the homotopy groups of its $\bG$-theory may be
  computed using the spectral sequence with $m$ sufficiently large:
  \[E_2^{s,t} = \H_{\rm BM,M}^{s-t}(\rmE\rmG^{gm,m}{\underset {\rmG} \times}\rmX, {\mathbb Z}(-t)) \Ra \pi_{-s-t}(\bG(\rmE\rmG^{gm,m}{\underset {\rmG} \times}\rmX)) \cong \pi_{-s-t}(\bG(\rmX, \rmG)\compl_{\rho_{\tilde \rmG, m}}). \]
The last isomorphism is from \cite[Theorem 5.6(i)]{CJ23} and $\H_{\rm BM,M}$ denotes Borel-Moore motivic cohomology, which identifies with the higher equivariant Chow groups. This identification shows that the $E_2^{s,t}=0$ for all but finitely many values of $s$ and $t$, once $-s-t$ is fixed.  Therefore, 
 the above spectral sequence converges strongly, and the $E_2$-terms, and hence the abutment, will be independent of $m$, if $m>>0$.
\vskip .2cm
We proceed to show the existence of a spectral sequence as above even when the scheme $\rmX$ is singular, provided it satisfies the 
assumption that it is $\rmG$-quasi-projective as in Definition ~\ref{G.quasiproj}. In this case, the hypothesis in Definition ~\eqref{G.quasiproj} shows that one 
may find a $\rmG$-equivariant closed immersion of $\rmX$ into a smooth scheme $\tilde \rmX$ with open complement $\tilde \rmU$. 
This provides the commutative diagram ~\eqref{exact.couple.0}, where $s_{\le q}\bK$ denotes the $q$-th slice of the spectrum $\bK$. 

\be \begin{equation} 
\label{exact.couple.0}
\xymatrix{{Map(\Sigma_{\rmS^1}\rmE\rmG^{gm,m}{\underset {\rmG} \times}\tilde \rmX_+, {\bH}({\mathbb Z}(q)[2q]))} \ar@<1ex>[r] \ar@<1ex>[d] & {Map(\Sigma_{\rmS^1}\rmE\rmG^{gm,m}{\underset {\rmG} \times}\tilde \rmU_+, {\bH}({\mathbb Z}(q)[2q]))} \ar@<1ex>[d]\\
          {Map(\Sigma_{\rmS^1}\rmE\rmG^{gm,m}{\underset {\rmG} \times}\tilde \rmX_+, s_{\le q}\bK )} \ar@<1ex>[r] \ar@<1ex>[d] & {Map(\Sigma_{\rmS^1}\rmE\rmG^{gm,m}{\underset {\rmG} \times}\tilde \rmU_+, s_{\le q}\bK)} \ar@<1ex>[d]\\
          {Map(\Sigma_{\rmS^1}\rmE\rmG^{gm,m}{\underset {\rmG} \times}\tilde \rmX_+, s_{\le q-1}\bK)} \ar@<1ex>[r]  & {Map(\Sigma_{\rmS^1}\rmE\rmG^{gm,m}{\underset {\rmG} \times}\tilde \rmU_+, s_{\le q-1}\bK)}}
\end{equation} \ee
Taking the homotopy fiber of the horizontal maps in each row now provides the fiber sequence:
\be \begin{multline}
  \begin{split}
  \label{exact.couple.1}
{Map(\Sigma_{\rmS^1}\rmE\rmG^{gm,m}{\underset {\rmG} \times}\tilde \rmX_+, \Sigma_{\rmS^1}\rmE\rmG^{gm,m}{\underset {\rmG} \times}\tilde \rmU_+, {\bH}({\mathbb Z}(q)[2q]))} \ra  {Map(\Sigma_{\rmS^1}\rmE\rmG^{gm,m}{\underset {\rmG} \times}\tilde \rmX_+, \Sigma_{\rmS^1}\rmE\rmG^{gm,m}{\underset {\rmG} \times}\tilde \rmU_+,  s_{\le q}\bK))} \\ 
\ra {Map(\Sigma_{\rmS^1}\rmE\rmG^{gm,m}{\underset {\rmG} \times}\tilde \rmX_+, \Sigma_{\rmS^1}\rmE\rmG^{gm,m}{\underset {\rmG} \times}\tilde \rmU_+, s_{\le q-1}\bK)}.
   \end{split}
\end{multline} \ee
The exact couple is provided by the above fiber sequence. The same reasoning as above in the case when $\rmX$ is smooth shows the spectral sequence converges strongly and has the required $E_2$-terms. This completes the proof of the Theorem. \qed
 \vskip .2cm
 \begin{remarks}
 {\rm For more general groups other than a split torus, there is still
  a similar spectral sequence (see \cite[Theorem 4.7]{CJ23}), but only for the full $\rmE\rmG^{gm}{\underset {\rmG} \times}\rmX$. The exact couple for this spectral sequence is obtained by taking the 
  homotopy inverse limit of the tower of fibrations in ~\eqref{exact.couple.1}. This spectral sequence converges only conditionally, and computes the homotopy groups of the full derived completion.}
 \qed
 \end{remarks}
 \subsection{Toric varieties, Spherical varieties, Linear varieties and proof of Corollary ~\ref{toric.sph} }
\label{Toric.Spherical.vars}
 \vskip .1cm
 Here a key observation is that {\it quasi-projective toric varieties} are defined to be {\it normal} quasi-projective varieties on
 which a split torus $\rmT$ acts with finitely many orbits. Therefore, by \cite[Theorem  1]{SumI} (and \cite[Theorem 2.5]{SumII},
 such toric varieties are $\rmT$-quasi-projective.
 This already includes many familiar varieties: see \cite{Oda}.
 \vskip .1cm
 Given a connected split reductive group $\rmG$, a {\it quasi-projective $\rmG$-spherical variety} is defined to be a normal quasi-projective $\rmG$-variety 
 on which $\rmG$, as well as a Borel subgroup of $\rmG$, have finitely many orbits. (See \cite{Tim}.) This class of varieties clearly includes all 
 toric varieties (when the group is a split torus), but also includes many other varieties. The discussion in
 the last paragraph above, now shows these are also $\rmG$-quasi-projective varieties. {\it  These observations  readily complete the proof of Corollary ~\ref{toric.sph}.}
  \vskip .2cm
  \begin{remarks} {\rm 
  The counter-example in \cite[6.1.1]{CJ23} shows that for large classes of toric and spherical varieties,
  the equivariant algebraic K-theory with finite coefficients away from the characteristic is not isomorphic to
  the corresponding equivariant K-theory with the Bott element inverted, so that Thomason's theorem does not provide
  a completion theorem for equivariant K-theory unless the Bott element is inverted. Neither does \cite[Theorem 1.2]{K} 
  unless these varieties are projective and smooth (and therefore, stratified by affine spaces). 
  \vskip .1cm
  Finally we recall that {\it linear varieties} are those varieties that have a stratification where each stratum is a product of an
  affine space and a split torus.  
  In view of the existence of a localization sequence for equivariant ${\bf G}$-theory and equivariant homotopy K-theory, 
  it is clear the $\rmG$-equivariant $\bG$-theory and homotopy K-theory
  of such varieties can be readily understood in terms of the corresponding $\rmG$-equivariant $\bG$-theory and homotopy K-theory of 
  the strata. 
  Examples of such varieties are normal quasi-projective varieties on which a connected split solvable group
  acts with finitely many orbits. (See \cite[p. 119]{Rosen}.)
  }
  \end{remarks}



\begin{thebibliography}{MMMMM}
\frenchspacing

\bibitem[AB]{AB}
M.~Atiyah, R.~Bott: {\em A Lefschetz fixed point theorem for elliptic operators},
Annals of Math., 86 (1967) 374-407, 87(1968) 451-491.


\bibitem[AS69]{AS69}
M.~Atiyah, G.~Segal: Equivariant K-theory and completion, J. Diff. Geom., \textbf{3}, (1969), 1-18.



\bibitem[Ch]{Ch} S\'eminaire C. Chevalley, 2e ann\'ee, {\it Anneaux de Chow et applications},  Paris: Secr\'etariat
math\'ematique, (1958).


\bibitem[CHH04]{CHH04} C. Haesmeyer, \emph{Descent properties for Homotopy K-theory}, Duke Math J,  \textbf{125}, (2004), 589-620.


\bibitem[CJ23]{CJ23} G. Carlsson and R. Joshua, \textit{Equivariant Algebraic K-Theory,  G-Theory and Derived completion}, Advances in Mathematics, \textbf{430}, 109194, (2023):
https://doi.org/10.1016/j.aim.2023.109194.

\bibitem[CJP23]{CJP23} G. Carlsson, R. Joshua and P. Pelaez, \textit{Equivariant Algebraic K-Theory and Derived completions II: The case of Equivariant Homotopy K-Theory and Equivariant K-Theory},
Preprint, (2023).


\bibitem[EG]{EG} D. Edidin and W. Graham, \textit{Equivariant Intersection theory}, Inventiones Math., \textbf{131}, 595-634, (1998).


\bibitem[Gi]{Gi} H. Gillet, \emph{Riemann-Roch theorems for Higher Algebraic K-theory}, Advances in Math., \textbf{40}, 203-289, (1981).



\bibitem[Jan]{Jan} U. Jannsen, \textit{Deligne Homology, Hodge-${\mathcal D}$-conjecture and motives}, pp. 305-372, in \textit{Beilinson conjectures on Special Values of L-functions}, Perspectives in Math, \textbf{4}, (1988), Academic Press.

\bibitem[J93]{J93} R. Joshua, \textit{The derived category and Intersection cohomology of algebraic stacks}, in \textit{Algebraic K-Theory and Algebraic Topology},
Lake Louise, AB, 1991. NATO Adv. Sci. Inst. Ser. C Math Phys. Sc., \textbf{407}, Kluwer Acad. Publ., Dordrecht, 1993, pp. 91-145.

\bibitem[J99]{J99} R. Joshua, \textit{Equivariant Riemann-Roch for $\rmG$-quasi-projective varieties}, K-Theory, \textbf{17}, no.1, (1999), 1-35.


\bibitem[J02]{J02} R. Joshua, \textit{Higher intersection theory on algebraic stacks}, I and II,
K-Theory, \textbf{27}, 134-195, and 197-244.

\bibitem[J10]{J10} R. ~Joshua, \textit{K-theory and G-theory of dg-stacks},
in {\it Regulators},  Contemp. Mathematics, \textbf{571}, American Math Soc, (2012), 175-217.


\bibitem[JK]{JK} R. ~Joshua and A. ~Krishna, \emph{Higher K-theory of toric stacks}, Ann. Scient. Norm. Sup. Pisa Cl. Sci. (5),
\textbf{XIV}, (2015), 1-41.

\bibitem[J20]{J20} R. Joshua, \textit{Equivariant Derived categories for Toroidal Group Imebddings}, Transformation Groups, \textbf{27}, issue 1, 113-162, (2022).

\bibitem[HHA]{HHA} H. H. Andersen, \textit{Schubert varieties and Demazure character formula}, Invent. Math., \textit{79}, 611-618, (1985).

\bibitem[K]{K} A. ~Krishna, \emph{The Completion problem for Equivariant K-theory}, Preprint, 
see arXiv:1201.5766v1 (2012), v2 (2015), J. Reine Angew. Math. 740 (2018), 275-317.


\bibitem [MV] {MV} F. Morel and V. Voevodsky, \textit{${\mathbb A}^1$-homotopy
theory of schemes}, I. H. E. S Publ. Math., {90}, (1999), 45-143 (2001).

\bibitem[Oda]{Oda} T. Oda, {\it Convex bodies and algebraic geometry}, An introduction to the theory of toric varieties. Translated from the Japanese. Ergebnisse der Mathematik und ihrer Grenzgebiete (3) [Results in Mathematics and Related Areas (3)], {\bf 15}. Springer-Verlag, Berlin, 1988.Erg. der Math., Springer,(1987). 


\bibitem[Rosen]{Rosen}M. Rosenlicht, \emph{Questions of rationality for solvable algebraic groups over nonperfect
fields}, Annali di Mat. 61 (1963), 97-120.

\bibitem[Sch]{Sch} M. Schlichting, {\it Negative K-theory of derived categories}, Math. Zeitscschrift, (2006), no.1, 97-134.

\bibitem[Sch11]{Sch11} M. Schlichting, {\it Higher Algebraic K-theory}, Sedano Lectures on the winter school in K-theory, Lect. Notes in Mathematics,
\textbf{2008}, (2011), 167-241.



\bibitem[SumI]{SumI} H. Sumihiro, \textit{Equivariant completion I}, J. Math. Kyoto Univ., \textbf{14},
(1974), 1-28.

\bibitem [SumII]{SumII} H. Sumihiro, \textit{Equivariant completion II}, J. Math. Kyoto Univ., \textbf{15-3},
(1975), 573-605.

\bibitem[Tim]{Tim} D. A. Timashev, \emph{Homogeneous spaces and Equivariant imbeddings}, 
Encyclopaedia of Mathematical Sciences, \textbf{138}, Springer, (2011).


\bibitem[Th86-1]{Th86-1}R. Thomason, \emph{Lefschetz-Riemann-Roch and the coherent trace formula}, Invent. math, \textbf{85}, (1986), 515-543.

\bibitem[Th86-2]{Th86-2} R. Thomason, \emph{Comparison of equivariant algebraic and topological K-theory},
Duke Math. J. 53 (1986), no. 3, 795-825. 

\bibitem[Th88]{Th88} R. Thomason, \textit{Equivariant Algebraic vs. Topological K-homology Atiyah-Segal style}, Duke Math J.,
\textbf{56, 3}, 589-636, (1988).



\bibitem[Tot]{Tot} B. Totaro, \textit{ The Chow ring of a classifying space},
Algebraic K-theory (Seattle, WA, 1997), 249-281, Proc. Symposia in Pure Math,
{\bf 67}, AMS, Providence, (1999). 


\bibitem[Wei81]{Wei81} C. Weibel, \textit{KV-theory of categories}, Trans. AMS., \textbf{267}, No. 2, (1981), 621-635, American Math Soc.

\bibitem[Wei89]{Wei89} C. Weibel, \textit{Homotopy Algebraic K-Theory}, Contemp. Math., \textbf{83}, (1989), 461-488, American Math Soc.


\end{thebibliography}
\end{document}